\documentclass[reqno]{amsart}

\usepackage{amscd,amsthm,amsmath,amssymb,mathtools}

\usepackage{upgreek,bm}

\usepackage{verbatim}
\usepackage{cite}
\usepackage{xspace}
\usepackage{xcolor}
\usepackage{url}
\usepackage{enumerate,enumitem}

\usepackage{graphicx}

 \usepackage{algorithm}
 \usepackage{algorithmic}

\usepackage{epstopdf,epsfig,subfigure}
\usepackage{curve2e}
\usepackage{setspace}

\usepackage[
    bookmarks=true, 
    unicode=false, 
    pdftoolbar=true,  
    pdfmenubar=true, 
    pdffitwindow=false,  
    pdfstartview={FitH}, 
    pdftitle={Stabilization with switching delta distridutions},    
    pdfauthor={Author},    
    pdfsubject={Subject},   
    pdfcreator={Creator},  
    pdfproducer={Producer}, 
    pdfkeywords={keyword1, key2, key3},
    pdfnewwindow=true,     
    colorlinks=true,      
    linkcolor=red,         
    citecolor=blue,        
    filecolor=magenta,      
    urlcolor=cyan,          
hypertexnames=false
]{hyperref}

\usepackage{setspace}
\usepackage{geometry}

\theoremstyle{plain}

\newtheorem{theorem}{Theorem}[section]

\newtheorem{lemma}[theorem]{Lemma}
\newtheorem{corollary}[theorem]{Corollary}

\newtheorem{assumption}[theorem]{Assumption}

\theoremstyle{definition}
\newtheorem{definition}[theorem]{Definition}
\newtheorem{remark}[theorem]{Remark}

\numberwithin{equation}{section}

\DeclareMathOperator{\proj}{Proj}

\newcommand{\sign}{\mathop{\rm sign}\nolimits}

\newcommand{\rest}{\left.\kern-2\nulldelimiterspace\right|_}
\newcommand{\norm}[2]{\left|#1\right|_{#2}}
\newcommand{\dnorm}[2]{\left\|#1\right\|_{#2}}

\newcommand{\zero}{{\mathbf0}}
\newcommand{\Id}{{\mathbf1}}

\newcommand{\ex}{\mathrm{e}}
\newcommand{\p}{\partial}

\newcommand{\deltafun}{\bm\updelta}

\newcommand*{\Bigcdot}{\raisebox{-.25ex}{\scalebox{1.25}{$\cdot$}}}

\newcommand{\clC}{{\mathcal C}}

\newcommand{\clF}{{\mathcal F}}
\newcommand{\clG}{{\mathcal G}}

\newcommand{\clI}{{\mathcal I}}
\newcommand{\clJ}{{\mathcal J}}

\newcommand{\clL}{{\mathcal L}}

\newcommand{\clU}{{\mathcal U}}
\newcommand{\clV}{{\mathcal V}}

\newcommand{\clX}{{\mathcal X}}

\newcommand{\bbN}{{\mathbb N}}

\newcommand{\bbR}{{\mathbb R}}

\newcommand{\bbT}{{\mathbb T}}

\newcommand{\fkK}{{\mathfrak K}}

\newcommand{\fkS}{{\mathfrak S}}

\newcommand{\fkY}{{\mathfrak Y}}

\newcommand{\rmD}{{\mathrm D}}

\newcommand{\rmY}{{\mathrm Y}}

\newcommand{\bfn}{{\mathbf n}}

\newcommand{\bfu}{{\mathbf u}}

\newcommand{\bfx}{{\mathbf x}}

\newcommand{\rmd}{{\mathrm d}}
\newcommand{\rme}{{\mathrm e}}

\newcommand{\rms}{{\mathrm s}}

\newcommand{\fkh}{{\mathfrak h}}

\newcommand{\fkv}{{\mathfrak v}}
\newcommand{\fkw}{{\mathfrak w}}

\newcommand{\ovlineC}[1]{\overline C_{\left[#1\right]}}

\definecolor{DarkBlue}{rgb}{0,0.08,0.45}
\definecolor{DarkRed}{rgb}{.65,0,0}
\definecolor{applegreen}{rgb}{0.55, 0.71, 0.0}
\definecolor{gray06}{rgb}{0.6, 0.6, 0.6}

\newcounter{mymac@matlab}
\newcommand{\matlab}{MATLAB%
   \ifnum\value{mymac@matlab}<1%
   \textregistered%
   \setcounter{mymac@matlab}{1}%
   \fi%
  }

\newcommand{\black}{ \color{black} }

\begin{document}
\title{Stabilizability of parabolic equations by switching controls based on point actuators}
\author{Behzad Azmi$^{\tt1}$}
\author{Karl Kunisch$^{\tt2,\tt3}$}
\author{S\'ergio S.~Rodrigues$^{\tt2}$}

 \thanks{
\vspace{-1em}\newline\noindent
\newline\noindent
$^{\tt1}$ Department of Mathematics and Statistics, University of Konstanz, D-78457 Konstanz, Germany \newline\noindent
 $^{\tt2}$  Johann Radon Institute for Computational and Applied Mathematics,
 \"OAW, 
  Altenbergerstrasse 69,
 4040 Linz, Austria.
    \newline\noindent
  $^{\tt3}$ Institute for Mathematics and Scientific Computing, University of Graz, Heinrichstrasse 36, 
  8010 Graz, Austria.
  \newline\noindent
  {\sc Emails}:
 {\small\tt behzad.azmi@uni-konstanz.de;\quad karl.kunisch@uni-graz.at;\\ \hspace*{3em} sergio.rodrigues@ricam.oeaw.ac.at}.
  \newline\noindent
%$^{*}$ Corresponding author.
}

\begin{abstract}
It is shown that a switching control involving  a finite number of Dirac delta actuators is able to steer the state of a general class of nonautonomous parabolic equations to zero as time increases to infinity.  The strategy  is based on a recent feedback stabilizability result,  which utilizes control forces given by linear combinations of appropriately located  Dirac delta distribution actuators.  Then, the existence of a stabilizing switching control with no more than one actuator active at each time instant is established.   For the implementation in practice,  the stabilization problem  is formulated as an infinite horizon optimal control problem, with cardinality-type control constraints enforcing the switching property.  Subsequently,  this  problem is tackled using a receding horizon framework.  Its suboptimality and stabilizabilizing properties are analyzed.  Numerical simulations validate the approach, illustrating its stabilizing and switching properties.
\end{abstract}

\maketitle

\pagestyle{myheadings} \thispagestyle{plain} \markboth{\sc B. Azmi, K. Kunisch, and S. S.
Rodrigues}{\sc Stabilization of parabolic equations by switching point controls}

\section{Introduction}
The stabilizability of a general class of parabolic-like equations by switching Diracs is investigated.  To set the stage,   we consider controlled parabolic equations of the form
\begin{subequations}\label{sys-y-axy}
 \begin{align}
 &\tfrac{\p}{\p t} y(t,x)-\nu\Delta y(t,x)+a(x,t)y(t,x)+b(x,t)\cdot\nabla y(t,x)=u(t)\deltafun_{c(t)},\\
 & y(0,x)=y_0(x),\qquad \clG y(t,x)\rest{\p\Omega}=0,
 \end{align}
 \end{subequations}
with state~$y(t,x)$ defined for~$x$ in an open convex, polygonal  domain~$\Omega\subset\bbR^d$, $d\in\{1,2,3\}$, and for time~$t>0$. The delta distribution~$\deltafun_{c(t)}$, supported at the point~$c(t)\in\Omega$,  represents a  switching actuator, where the mapping~$t\to c(t)\in\bfx\subset\Omega$ is piecewice constant,  and~$\bfx\coloneqq\{x^1,x^2,\dots, x^{M}\}$ is a finite set containing the~$M\in\bbN$ pairwise distinct supports~$x^j\in\Omega$ of the actuators in~$\{\deltafun_{x^j}\mid 1\le j\le M\}$.
The goal consists in establishing   the existence of a stabilizing  control~$(u(t),c(t))\in \bbR\times\bfx$ and also its numerical realization, where
 \begin{align}\label{goal-Intro-u}
(u(t),c(t))\in \bbR\times\bfx\quad\mbox{with}\quad u\in L^2((0,+\infty),\bbR)\quad\mbox{and}\quad\mbox{$c$ piecewise constant,}
 \end{align}
so that the solution of~\eqref{sys-y-axy}
satisfies, for some constants~$C\ge1$ and~$\mu>0$,
\begin{align}\label{goal-Intro-exp}
 &\norm{y(t,\Bigcdot)}{W^{-1,2}(\Omega)}\le C\ex^{-\mu t}\norm{y_0}{W^{-1,2}(\Omega)},\quad\mbox{for all}\quad t\ge0.
  \end{align}
We thus focus on the stabilizability by finitely many Dirac measures with only one  active at any given time (switching property).  
 The nonswitching case was treated in~\cite{KunRodWal24-cocv} where it was proven  that  for sufficiently many actuators, depending on ~$(\nu,a,b)$,  stabilizability holds  with an explicitly given input feedback control, with control forces taking values in the linear span of~$\deltafun_\bfx$.
In~\eqref{sys-y-axy}, the operator~$\clG$ defines the boundary conditions. Our results are valid for Dirichlet, Neumann, and periodic boundary conditions (bcs),
\begin{align}
\clG&=\clG_{\rm Dir}=\Id,&&\quad\mbox{for Dirichlet bcs}\notag\\
\clG&=\clG_{\rm Neu}=\bfn\cdot\nabla,&&\quad\mbox{for Neumann bcs}\notag\\
\clG&=\clG_{\rm Per}=\zero,&&\quad\mbox{for periodic bcs}\notag
\end{align}
where~$\Id$ stands for the identity operator, $\bfn=\bfn(\overline x)$ stands for the unit outward normal vector at the smooth  points~$\overline x\in\p\Omega$ of the boundary, and~$\clG=\clG_{\rm Per}$ simply means that the boundary conditions are inexistent, by considering  the evolution on the boundaryless torus~$\Omega=\bbT_L^d\sim\bigtimes_{n=1}^d[0,L_n)$;~$\p\bbT_L^d=\emptyset$.
We shall also assume that
\begin{subequations}\label{assum.ab.parab}
\begin{align}
 & a\in L^\infty((0,+\infty), \clC^1(\overline\Omega))\quad\mbox{and}\quad b\in L^\infty((0,+\infty), \clC^2(\overline\Omega))^d,
\end{align}
and further that
\begin{align}
b\cdot\bfn=0,\quad\mbox{if}\quad\clG=\clG_{\rm Neu}.
\end{align}
\end{subequations}
As usual, we shall consider~$H=L^2(\Omega)$ as a pivot space and introduce the following Sobolev subspace~$V$, depending on the boundary conditions,
\begin{subequations}\label{bcs.parab}
\begin{align}
V&=V_{\rm Dir}=\{h\in W^{1,2}(\Omega)\mid h\rest{\p\Omega}=0\},&&\quad\mbox{for Dirichlet bcs}\\
V&=V_{\rm Neu}=W^{1,2}(\Omega),&&\quad\mbox{for Neumann bcs}\\
V&=V_{\rm Per}=W^{1,2}(\bbT_L^d),&&\quad\mbox{for periodic bcs}
\end{align}
and denote the continuous dual of~$V$ by~$V'$. Hence $V\subset H=H'\subset V'$.
\end{subequations}

%%%%%%%%%%%%%%%%%%%%%%%
%%%%%%%%%%%%%%%%%%%%%%%
\subsection{Main stabilizability result}
We shall prove a stabilizability result for a class of abstract evolution equations, which we shall subsequently apply to the concrete system~\eqref{sys-y-axy} to obtain the following. 
\begin{theorem}\label{T:main-Intro-swi}
 Let~$\Omega\subset\bbR^d$, $d\in\{1,2,3\}$, be an open, convex, polygonal domain, let
 $(a,b)$ satisfy~\eqref{assum.ab.parab}, and let $\mu >0$. Then,
 there exist a constants  $C>0$, $C_u>0$,  and a finite set~$\bfx=\{x^i\mid 1\le i\le M\} \subset \Omega$, $M\in\bbN$, such that for each~$y_0\in V'$, there exists a control~$u\in L^2((0,+\infty),\bbR)$,  and a piecewise constant path $c(t)\in\bfx$
such that the solution of~\eqref{sys-y-axy}
 satisfies
 \begin{align}\label{goal.parab-Intro-exp}
 &\norm{ y(t,\Bigcdot)}{V'}\le C\ex^{-\mu t}\norm{y_0}{V'},\qquad\mbox{for all}\qquad t\ge0.
\end{align}
Furthermore, the mapping~$y_0\mapsto u(y_0)$ is bounded from~$V'$
into~$L^2((0,+\infty),\bbR)$,  and  we have  $\norm{u(y_0)}{L^2((0,+\infty),\bbR)}\le C_u\norm{y_0}{V'}$.
Finally, for fixed~$\nu>0$, the number of actuators depends only on an upper bound for the operator norm~$\norm{a\Id+b\cdot\nabla}{\clL(H,V')}$
\end{theorem}

\begin{remark}
The auxiliary (nonswitching) stabilizability result from~\cite{KunRodWal24-cocv} that we use, was applied  to rectangular/box domains in~\cite{KunRodWal24-cocv}, where it is also mentioned that it can be applied to some classes of convex domains; see~\cite[Rem.~5]{KunRodWal24-cocv}. In fact, it can be applied to general convex polygonal/polyhedral domains~$\Omega\subset\bbR^d$ as explained in~\cite[Rem.~2.8]{AzmiKunRod23}, using the fact that such domains are the union of a finite number of simplexes and each simplex can be successively partitioned into smaller rescaled simplexes while keeping the number of used congruent classes bounded; see~\cite[Sect.~3]{EdelsbrunnerGrayson00}  \cite[Thm.~4.1]{Bey00}.
\end{remark}

\begin{remark}
In the literature,  {\em stabilizability} often entails the property that~\eqref{goal.parab-Intro-exp} is achieved by controls  in feedback form~\cite[Sects.~I.2.5 and~IV.3.3]{Zabczyk92}. In this manuscript this is not the case, since it is  not required that the control input pair~$(u(t),c(t))$ is given as a function of the time-state pair~$(t,y(t))$. Following~\cite[Sect.~5.5]{Sontag98},  {\em stabilizability by means of nonfeedback controls} can be seen as  {\em asymptotic null controllability}.
\end{remark}

%%%%%%%%%%%%%%%%%%%%%%%
%%%%%%%%%%%%%%%%%%%%%%%
\subsection{Computational approach}
We also address the computation of a stabilizing switching control, which is a  nontrivial task. Our approach will be as follows.
We consider the infinite time-horizon optimization problem
\begin{equation}
\notag
\frac{1}{2} \int^{\infty}_{0} \left(   \norm{y(t)}{L^2(\Omega)}^2+\beta \norm{\mathbf{u}(t)}{\ell^2}^2\right)dt \quad\longrightarrow\quad {\rm infimum},
\end{equation}  
 with~$\bfu\in L^2((0,+\infty),\bbR^M)$ and where~$(y,\bfu)$ is subject  to~\eqref{sys-y-axy},
 \begin{align*}
&\tfrac{\p}{\p t} y-\nu\Delta y+ay+b\cdot\nabla y=\sum^{M}_{j =1} \bfu_j \deltafun_{x^j}, \qquad
 y(0)=y_0,\qquad \clG y\rest{\p\Omega}=0,
\intertext{and to  cardinality control constraints}
& \qquad   \norm{\bfu(t)}{0}\coloneqq\sum^{M}_{j =1} \sign(\norm{\bfu_j(t)}{\bbR})\le 1,
\end{align*}
where~$\sign$ denotes the sign function, thus~$\sign(0)\coloneqq 0$ and~$\sign(r)\coloneqq 1$ for $r>0$.  Hence,~$\norm{\bfu(t)}{0}\in\{0,1,\dots,M\}$ counts the number of nonzero components of~$\bfu(t)$. The details to deal with such nonsmooth  and nonconvex optimization problem shall be presented in Section~\ref{S:stabil_rhc}, where we shall also employ a receding horizon strategy to approximate the  infinite time-horizon problem. Above,~$\norm{\Bigcdot}{\ell^2}$ denotes the Euclidean norm in~$\bbR^M$,~$\norm{v}{\ell^2}\coloneqq(\sum_{j=1}^M v_j^2)^\frac12$.

%%%%%%%%%%%%%%%%%%%%%%%
%%%%%%%%%%%%%%%%%%%%%%%
\subsection{Related literature}
There are relatively few results in the literature on stabilizability  by means of switching controls, though it is an important problem for applications where  it is expensive to have several actuators active at the same time. Concerning switching controls with focus on controllability results    we can mention~\cite{Zuazua11},  where Fourier series representations of solutions are used to tackle, among other cases,  1D heat equation with  controls switching between two delta actuators.  We also  mention~\cite{Gugat08} where controllability results for 1D wave equations are investigated with switching boundary controls,  and \cite{LamGirPri15} which explores the exponential stabilizability of feedback switching boundary controls for a class of switched linear hyperbolic systems of conservation laws.   In \cite{HanGun10},    feedback switching controls are introduced  for a specific class of semilinear hyperbolic systems within an augmented bounded variation setting.
To verify the stabilizability by switching controls, here we follow in part appropriate variations of arguments as in~\cite{AzmiKunRod21}, where indicator functions of small open spatial subdomains, instead of delta distributions, were taken as actuators, and where the continuity of solutions as the control varies in a so-called relaxation metric~\cite{Gamk78} is used.

Concerning numerical computations, the followed receding horizon framework, also known as model predictive control, has garnered significant attention from researchers, due to its flexibility, namely, in tackling high-dimensional problems with complex structures~\cite{VeldZua22, VeldBorZua24, Gru09},  as well as control and state constraints~\cite{GruMecPirVol21,AzmiKunRod21,AzmiKunRod23}.

Achieving the desired switching property is often enhanced/pursued either through continuous optimization algorithms incorporating (squared) $\ell_1$-type regularization \cite{ClaRunKun16,ClaRunKun17, AzmiKunisch19}   or via relaxation techniques for mixed-integer optimization problems~\cite{HanSag13}.  Here,  we utilize a continuous optimization algorithm  incorporating a gradient projection method inspired by~\cite{Wachs24,BecEld13}; this  approach ensures that no more that one actuator is active at each time instant.

%%%%%%%%%%%%%%%%%%%%%%%
%%%%%%%%%%%%%%%%%%%%%%%
\subsection{Notation}
Given Hilbert spaces~$X$ and~$Y$, the space of continuous linear mappings from~$X$ into~$Y$ is denoted by~$\clL(X,Y)$, and for ~$X=Y$ we
write~$\clL(X)\coloneqq\clL(X,X)$.
The continuous dual of~$X$ is denoted~$X'\coloneqq\clL(X,\bbR)$. The adjoint of an operator $L\in\clL(X,Y)$ will be denoted $L^*\in\clL(Y',X')$.
If the inclusion
$X\subseteq Y$ is continuous, we write $X\xhookrightarrow{} Y$. We write
$X\xhookrightarrow{\rm d} Y$, respectively $X\xhookrightarrow{\rm c} Y$, if the inclusion is also dense, respectively compact.

We write~$\bbR$ and~$\bbN$ for the sets of real numbers and nonnegative
integers, respectively, and  we denote by $\bbR_+\coloneqq(0,+\infty)$ and~$\bbN_+\coloneqq\mathbb N\setminus\{0\}$, their subsets with positive numbers.

For an open interval~$I\subset\bbR$, we denote~$W(I,X,Y)\coloneqq\{f\in L^2(I,X)\mid \dot f\in L^2(I,Y)\}$, endowed with the norm~$\norm{f}{W(I,X,Y)}\coloneqq \norm{(f,\dot f)}{L^2(I,X)\times L^2(I,Y)}$.

The space of continuous functions from~$X$ into~$Y$ is denoted by~$\clC(X,Y)$.
The unit sphere in~$X$ shall be denoted by
$\fkS_X\coloneqq\{h\in X\mid \norm{h}{X}= 1\},$
where~$\norm{\Bigcdot}{X}$ denotes the norm on~$X$, associated with the scalar product~$(\Bigcdot,\Bigcdot)_X$.

By~$\overline C_{\left[a_1,\dots,a_n\right]}$ we denote a nonnegative function that
increases in each of its nonnegative arguments~$a_i\ge0$, $1\le i\le n$.
Finally, $C,\,C_i$, $i=0,\,1,\,\dots$, stand for unessential positive constants, which may take different values at different places in the manuscript.

\subsection{Contents}
To analyze  \eqref{sys-y-axy}, it will be considered as a special case of the abstract  parabolic-like equation
\begin{align}\label{sys-y-Intro}
 \dot y+Ay+A_{\rm rc}(t)y=u(t)\Phi(t),\qquad y(0)=y_0,\qquad t>0.
\end{align}
This allows to simplify parts of the exposition. It also facilitates   writing key properties as abstract assumptions, which makes the results potentially applicable to other concrete systems.
In Section~\ref{S:assumptions} we present  the assumptions that are required  on  the operators~$A$ and~$A_{\rm rc}$ above.
The main abstract exponential stabilizability result is proven in Section~\ref{S:switching}. This result is applied to concrete parabolic equations as~\eqref{sys-y-axy} in Section~\ref{S:proofT:main-Intro}, leading to the proof of Theorem~\ref{T:main-Intro-swi}.   Building upon the exponential stabilizability result of  Theorem~\ref{T:main-Intro-swi},  in Section~\ref{S:stabil_rhc} stabilizability of the receding horizon framework is investigated. Results of simulations are presented in Section~\ref{S:num_impl}, showing the algorithmic behavior of the proposed  methodologies and the stabilizing performance of the switching receding horizon control. Finally, Section~\ref{S:finalcomm} gathers concluding remarks on the achieved results and on interesting subjects for potential follow-up studies.

%%%%%%%%%%%%%%%%%%%%%%%%%%%
%%%%%%%%%%%%%%%%%%%%%%%%%%%
%%%%%%%%%%%%%%%%%%%%%%%%%%%
\section{Assumptions}\label{S:assumptions}
The result will be achieved under general assumptions on the
operators~$A$ and~$A_{\rm rc}$ defining the free dynamics, and on a particular stabilizability
assumption of~\eqref{sys-y-Intro} by means of controls
based on linear combinations of a large enough finite number~$M$ of  Dirac delta actuators.
Hereafter, all Hilbert spaces are assumed real and separable.

Let us be given two Hilbert spaces~$V$ and~$H$, with~$V\subset H=H'$.\black
\begin{assumption}\label{A:A0sp}
 $A\in\clL(V,V')$ is symmetric and $(y,z)\mapsto\langle Ay,z\rangle_{V',V}$ is a complete scalar product on~$V.$
\end{assumption}

Hereafter, we suppose that~$V$ is endowed with the scalar product~$(y,z)_V\coloneqq\langle Ay,z\rangle_{V',V}$,
which still makes~$V$ a Hilbert space.
Necessarily, ~$A\colon V\to V'$ is an isometry.
\begin{assumption}\label{A:A0cdc}
The inclusion $V\subseteq H$ is dense, continuous, and compact.
\end{assumption}

Necessarily, we have that
\begin{equation}\notag
 \langle y,z\rangle_{V',V}=(y,z)_{H},\quad\mbox{for all }(y,z)\in H\times V,
\end{equation}
and also that the operator $A$ is densely defined in~$H$, with domain $\rmD(A)$ satisfying
\begin{equation}\notag
\rmD(A)\xhookrightarrow{\rm d,\,c} V\xhookrightarrow{\rm d,\,c} H\xhookrightarrow{\rm d,\,c} V'\xhookrightarrow{\rm d,\,c}\rmD(A)'.
\end{equation}
Further,~$A$ has a compact inverse~$A^{-1}\colon H\to H$, and we can find a nondecreasing
system of (repeated accordingly to their multiplicity) eigenvalues $(\alpha_n)_{n\in\bbN_+}$ and a corresponding complete basis of
eigenfunctions $(e_n)_{n\in\bbN_+}$:
\begin{equation}\label{eigfeigv}
0<\alpha_1\le\alpha_2\le\dots\le\alpha_n\le\alpha_{n+1}\to+\infty \quad\mbox{and}\quad Ae_n=\alpha_n e_n.
\end{equation}

We can define, for every $\zeta\in\bbR$, the fractional powers~$A^\zeta$, of $A$, by
\begin{equation}\notag
 y=\sum_{n=1}^{+\infty}y_ne_n,\quad A^\zeta y=A^\zeta \sum_{n=1}^{+\infty}y_ne_n\coloneqq\sum_{n=1}^{+\infty}\alpha_n^\zeta y_n e_n,
\end{equation}
and the corresponding domains~$\rmD(A^{|\zeta|})\coloneqq\{y\in H\mid A^{|\zeta|} y\in H\}$, and
$\rmD(A^{-|\zeta|})\coloneqq \rmD(A^{|\zeta|})'$.
We have that~$\rmD(A^{\zeta})\xhookrightarrow{\rm d,\,c}\rmD(A^{\zeta_1})$, for all $\zeta>\zeta_1$,
and we can see that~$\rmD(A^{0})=H$, $\rmD(A^{1})=\rmD(A)$, $\rmD(A^{\frac{1}{2}})=V$.

For the time-dependent operator we assume the following:
\begin{assumption}\label{A:A1}
For almost every~$t>0$ we have~$A_{\rm rc} \in\clL(H,V')$,
with a uniform bound as~$\norm{A_{\rm rc}}{L^\infty(\bbR_+,\clL(H,V'))}\eqqcolon C_{\rm rc}<+\infty.$
\end{assumption}

For given~$T\in\bbR_+$ and~$k\in\bbN_+$, we denote the time interval
\begin{equation}\label{IkT}
I_k^T\coloneqq (kT,kT+T).
\end{equation}
We will also need the following norm squeezing property, by means of controls based on linear combinations of~$M$ given actuators. 

\begin{assumption}\label{A:MstaticAct} 
There exist:
\begin{enumerate}[label={\rm(\alph*)}]
 \item a positive integer~$M$; positive real numbers~$T>0$ and~$\theta\in(0,1)$;
 \item\label{A:MstaticAct-norm} a linearly independent family~$\{\Phi_j\mid j\in\{1,2,\dots,M\}\}\subset\fkS_{\rmD(A)'}$; and
\item\label{A:MstaticAct-linfu} a family of
 functions~$\{v_k\in\clL(V',L^\infty(I_k^T,\bbR^M))\mid k\in\bbN\}$, satisfying the inequality
 $\sup\limits_{k\in\bbN}\norm{v_k}{\clL(V',L^\infty(I_k^T,\bbR^M))}\le\fkK$,
 \end{enumerate}
 such that:  for all $k\in\bbN$ the solution of
 \begin{align}\label{sys-y-static}
 \dot y+Ay+A_{\rm rc}(t) y=\textstyle\sum\limits_{j=1}^M(v_{k,j}(\fkv))(t)\Phi_j,\qquad w(kT)=\fkv,
\end{align}
with~$t\in I_k^T$, satisfies
\begin{align}\label{goal-static}
 &\norm{y(kT+T)}{V'}\le \theta\norm{\fkv}{V'},\qquad\mbox{for all}\quad\fkv\in V'.
 \end{align}
\end{assumption}

\begin{remark}\label{R:satAssum1}
Assumptions~\ref{A:A0sp}--\ref{A:A1} are satisfied for a general class of parabolic equations as~\eqref{sys-y-axy}. The satisfiability of Assumption~\ref{A:MstaticAct} is nontrivial, and shall be proven for equations~\eqref{sys-y-axy} evolving in convex polygonal/polyhedral domains~$\Omega\subset\bbR^d$, by combining a result from~\cite{KunRodWal24-cocv} together with tools from optimal control.
\end{remark}

\begin{remark}\label{R:muTtheta}
The requirements within Assumption~\ref{A:MstaticAct} will be sufficient to conclude the desired stabilizability with exponential rate~$\mu_{T,\theta}\coloneqq\frac1T\log(\tfrac{5}{2+3\theta})>0$, with switching controls (cf.~Thm.~\ref{T:main-switch-asy}). To achieve the complete stabilizability (i.e., with~$\mu$ fixed apriori) as stated in Theorem~\ref{T:main-Intro-swi} we shall need an extra argument (cf.~Sect.~\ref{sS:satAMstaticAct}), where we shall use the (not necessarily switching) complete stabilizability result from~\cite{KunRodWal24-cocv}.
\end{remark}

%%%%%%%%%%%%%%%%%%%%%%%%%%%
%%%%%%%%%%%%%%%%%%%%%%%%%%%
%%%%%%%%%%%%%%%%%%%%%%%%%%%
\section{Stabilizability with a switching control}\label{S:switching}
Recall that we want to show stabilizability by means of switching controls, where at each time instant, only one of the actuators~$\Phi_j\in\rmD(A^{-1})$ is active.

It is convenient to work with more ``regular''/``standard'' solutions which we will have if the actuators were in the pivot space~$H$. For this reason, we borrow the idea from~\cite[Proof of Thm~3.1]{KunRodWal24-cocv} and consider the ``regularized'' state
$w\coloneqq A^{-1}y$. In other words, in this section, instead of looking for a switching control stabilizing~\eqref{sys-y-Intro} in the~$V'$-norm,
\begin{align}\label{sys-y}
&\dot y+Ay+A_{\rm rc}(t)y=u(t)\Phi(t),&& y(0)=y_0\in V',
\end{align}
we shall be looking for a switching control stabilizing~\eqref{sys-w} in the~$V$-norm,
\begin{align}\label{sys-w}
&\dot w+Aw+A_{\rm rc}^A(t)w=u(t)A^{-1}\Phi(t),&& w(0)=w_0\in V,\\
&\mbox{where}\quad
A_{\rm rc}^A(t)\coloneqq A^{-1}A_{\rm rc}(t)A.\notag
\end{align}
Clearly, this two problems are equivalent, but now the ``actuators''~$A^{-1}\Phi_j\in H$ are in the pivot space and we can use more standard arguments concerning the solutions.

%%%%%%%%%%%%%%%%%%%%%%%%%%%
%%%%%%%%%%%%%%%%%%%%%%%%%%%
\subsection{Auxiliary results}\label{sS:auxil-switch}
We shall use some auxiliary results, which we state in this section.

\begin{lemma}\label{L:ArcA}
Let~$A_{\rm rc}$ and~$C_{\rm rc}$ be as in Assumption~\ref{A:A1}. Then, the operator~$A_{\rm rc}^A(t)\coloneqq A^{-1}A_{\rm rc}(t)A$ satisfies~$\norm{A_{\rm rc}^A}{L^\infty(\bbR_+,\clL(\rmD(A),V))}\le C_{\rm rc}$.
\end{lemma}
\begin{proof}
We have~$\norm{A_{\rm rc}^A(t)h}{V}=\norm{A_{\rm rc}(t)Ah}{V'}\le C_{\rm rc}\norm{A^{1}h}{H}=C_{\rm rc}\norm{h}{\rmD(A)}$.
\end{proof}

Let us  take a more general forcing~$f$ in place of the control, and consider the  system
\begin{align}\label{sys-w-fIk}
 &\dot w+Aw+A_{\rm rc}^A w=f,\qquad w(kT)=\fkw,\qquad t\in I_k^T,
 \intertext{in the time interval~$I_k^T= (kT,kT+T)$. The solution of~\eqref{sys-w-fIk} will be denoted by}
 &\fkY_k(\fkw,f)(t)\coloneqq w(t).\notag
\end{align}

\begin{lemma}\label{L:strg-sol-w}
Given~$k\in\bbN$, $T>0$, and~$(\fkw,f)\in V\times L^2(I_k^T,H)$, there exists one, and only one, solution~$w\in W(I_k^T,\rmD(A),V)$ for system~\eqref{sys-w-fIk}. Moreover, for a constant~$D_Y>0$, independent of~$k\in\bbN$, there holds
\begin{equation}\notag
\norm{w(t)}{V}^2\le D_Y\left(\norm{\fkw}{V}^2+\norm{f}{L^2(I_k^T,H)}^2\right),\quad t\in I_k^T.
\end{equation}
\end{lemma}

\begin{proof}
The proof follows by slight variations of standard arguments. Since, it is often given for a different class of  operators when compared to that in Lemma~\ref{L:ArcA}, namely, for~$A_{\rm rc}^A\in L^\infty(\bbR_+,\clL(V,H))$ and since we need to recall some of these arguments later in Section~\ref{sS:satAMstaticAct}, we provide a few details here.
The following apriori-like estimates also hold for standard Galerkin approximations of the solution based on spaces spanned by a finite number of eigenfunctions~$e_i\in\rmD(A)$ of~$A$, from which we can then derive the existence of the solution as a weak limit of such approximations.

Multiplying the dynamics in~\eqref{sys-w-fIk} by~$2Aw$, and using Lemma~\ref{L:ArcA}, we find
\begin{align}
\tfrac{\rmd}{\rmd t}\norm{w}{V}^2&= -2\norm{w}{\rmD(A)}^2-2(A_{\rm rc}^A w,Aw)_H+2(f,Aw)_H\notag\\
&\le -2\norm{w}{\rmD(A)}^2+2C_{\rm rc}\norm{w}{\rmD(A)}\norm{w}{V}+2\norm{f}{H}\norm{w}{\rmD(A)}\notag\\
&\le -\norm{w}{\rmD(A)}^2+2C_{\rm rc}^2\norm{w}{V}^2+2\norm{f}{H}^2.\label{strg-sol-dw}
\end{align}
Hence we can take~$D_Y=2\rme^{2C_{\rm rc}^2T}$.
Further, time integration of~\eqref{strg-sol-dw}  yields
\begin{align}
\label{estimate_BA}
\norm{w(kT+T)}{V}^2+\norm{w}{L^2(I_k^T,\rmD(A))}^2\le \norm{w(kT)}{V}^2+2C_{\rm rc}^2\norm{w}{L^2(I_k^T,V)}^2+2\norm{f}{L^2(I_k^T,H)}^2,
\end{align}
which implies~$w\in L^2(I_k^T,\rmD(A))$. Then, from the dynamics~\eqref{sys-w-fIk}, we find that
$\dot w\in L^2(I_k^T,H)$, hence,~$w\in W(I_k^T,\rmD(A),H)$.
Finally, the uniqueness of the solution follows from the linearity of the dynamics with~$(\fkw,f)=(0,0)$.
\end{proof}

From~\eqref{sys-y-static}, by denoting~$w\coloneqq Ay$, we arrive at the dynamical system
\begin{equation}\label{sys-w-Phi}
 \begin{split}
 &\dot w+Aw+A_{\rm rc}^A w=\textstyle\sum\limits_{j=1}^M(\widehat v_{k,j}(\fkw))A^{-1}\Phi_j,\qquad w(kT)=\fkw,\qquad t\in I_k^T,\\
 \mbox{with}\quad&\fkw=A^{-1}\fkv,\qquad\widehat v_k(\fkw)\coloneqq v_k(A\fkw),\qquad\norm{A^{-1}\Phi_j}{H}=1.
 \end{split}
\end{equation}

\begin{corollary}\label{C:MstaticAct}
Let~$M$, $T>0$,~$\theta\in(0,1)$, $\Phi_j$, $v_k$, and~$\fkK$, be as in Assumption~\ref{A:MstaticAct}. Then there hold:
\begin{itemize}
 \item the family~$\{A^{-1}\Phi_j\mid j\in\{1,2,\dots,M\}\}\subset\fkS_{H}$ is linearly independent,
\item the family~$\{\widehat v_k\in\clL(V,L^\infty(I_k^T,\bbR^M))\mid k\in\bbN\}$,  satisfies $\sup\limits_{k\in\bbN}\norm{\widehat v_k}{\clL(V,L^\infty(I_k^T,\bbR^M))}\le\fkK$,
 \end{itemize}
Further, for all $k\in\bbN$ the solution of~\eqref{sys-w-Phi}
satisfies
\begin{align*}
 &\norm{w(kT+T)}{V}\le \theta\norm{\fkw}{V},\qquad\mbox{for all}\quad\fkw\in V.
 \end{align*}
\end{corollary}
\begin{proof}
The vectors~$A^{-1}\Phi_j$ are linearly independent, due to the linear independence of the~$\Phi_j$.
To conclude the proof, we observe that~$\norm{\widehat v_k(\fkw)(t)}{\ell^2}=\norm{v_k(A\fkw)(t)}{\ell^2}\le\fkK\norm{A\fkw}{V'}=\fkK\norm{\fkw}{V}$ and~$\norm{w(kT+T)}{V}=\norm{Aw(kT+T)}{V'}\le \theta\norm{\fkv}{V'}=\theta\norm{\fkw}{V}$.
\end{proof}

%%%%%%%%%%%%%%%%%%%%%%%%%%%
%%%%%%%%%%%%%%%%%%%%%%%%%%%
\subsection{Norm squeezing with switching controls}\label{sS:switching}
The next result shows that we can squeeze the norm of the state in each time interval $I_k^T$.
\begin{corollary}\label{C:main-switch-sq}
There exists a switching  and piecewise constant control
\begin{equation}\notag
\clV^{k,\rm swi}(t)=u^{k,\rms}(t)\widehat\Phi^{k,\rms}(t),\qquad\widehat\Phi^{k,\rms}(t)\in\{\widehat\Phi_j\mid 1\le j\le M\},
\end{equation}
with~$\norm{u^{k,\rms}(t)}{\bbR}\le M\fkK\norm{\fkw}{V}$. Here, all the intervals of constancy are larger than some~$\tau>0$,  which can be taken independently of~$(k,\fkw)$, such that for the solution of~\eqref{sys-w-fIk} with~$f=\clV^{k,\rm swi}$, we have the norm squeezing  property
\begin{align}\notag
  \norm{\fkY_k(\fkw,\clV^{k,\rm swi})(kT+T)}{V}&\le\tfrac{2+3\theta}{5}\norm{\fkw}{V},\notag
  \end{align}
with~$\theta<1$ as in Corollary~\ref{C:MstaticAct}.
\end{corollary}
\begin{proof}
Note that Assumptions~\ref{A:A0sp} and~\ref{A:A0cdc} correspond to~\cite[Assums.~2.1 and~2.2]{AzmiKunRod21} and Corollary~\ref{C:MstaticAct} corresponds to~\cite[Assum.~2.4]{AzmiKunRod21}.

We cannot apply immediately the results in~\cite{AzmiKunRod21} because in~\cite[Assum.~2.3]{AzmiKunRod21} it is assumed that~$A_{\rm rc} \in\clL(V,H)$, while in Assumption~\ref{A:A1} we have~$A_{\rm rc} \in\clL(H,V')$.
However,   both Assumptions lead to the inequality in Lemma~\ref{L:strg-sol-w}; see the analogue inequality in~\cite[Eq.~(3.16)]{AzmiKunRod21}. It is this inequality that is essential within~\cite[Proof of Thm.~3.3, Sects.~3.1--3.5]{AzmiKunRod21}, where the fact $A_{\rm rc} \in\clL(V,H)$ is not explicitly used. This means that, denoting our control in~\eqref{sys-w-Phi} as
 \begin{align}
\clV^0(\fkw)(t)\coloneqq \textstyle\sum\limits_{j=1}^M(\widehat v_{k,j}(\fkw))\widehat\Phi_j,\quad\mbox{with}\quad\widehat\Phi_j\coloneqq A^{-1}\Phi_j,
  \end{align}
we can follow~\cite[Proof of Thm.~3.3]{AzmiKunRod21} and construct a sequence of  controls~$\clV^1(\fkw)$, $\clV^2(\fkw)$, $\clV^3(\fkw)$ and~$\clV^4(\fkw)$, where, for all~$t\in {I_{k}^T}$ and~$n\in\{1,2,3,4\}$,
 \begin{align}\label{diff_Vswitch}
 \Xi_n\coloneqq \norm{\fkY_k(\fkw,\clV^{n}(\fkw))(kT+T)-\fkY_k(\fkw,\clV^{n-1}(\fkw))(kT+T)}{V}\le \tfrac{1-\theta}{10}\norm{\fkw}{V},
  \end{align}
See the analogue equations in~\cite[Eqs.~(3.21), (3.40b), (3.43a), (3.58)]{AzmiKunRod21}. Furthermore, see~\cite[Eq.~(3.43c)]{AzmiKunRod21}, the control~$\clV^4(\fkw)$ is switching and piecewise constant taking values in the set~$\{r\widehat\Phi_j\mid\norm{r}{\bbR}\le\varSigma(\fkw)\}$, for some positive constant~$\varSigma(\fkw)\le M\fkK\norm{\fkw}{V}$, where the intervals of constancy are nondegenerate. Namely, by construction the actuators are activated as a finite number of repetitions of the cycle
\begin{equation}\label{cycle}
\begin{split}
& \widehat\Phi_1\to\widehat\Phi_2\to\dots\to\widehat\Phi_{M-1}\to\widehat\Phi_M
 \to\widehat\Phi_M\to\widehat\Phi_{M-1}\to\dots\to\widetilde\Phi_2\to\widehat\Phi_1,
 \end{split}
\end{equation}
where each actuator is active for a period of time not smaller than a suitable~$\tau>0$ independent of~$(k,\fkw)$; see~\cite[Eq.~(3.57b)]{AzmiKunRod21}.
Finally, by~\eqref{diff_Vswitch} it follows that
\begin{align}\notag
  \norm{\fkY_k(\fkw,\clV^{4}(\fkw))(kT+T)}{V}&\le \Xi_4+\Xi_3+\Xi_2+\Xi_1+\norm{\fkY_k(\fkw,\clV^{0}(\fkw))(kT+T)}{V}\notag\\
  &\le (4\tfrac{1-\theta}{10}+\theta)\norm{\fkw}{V}=\tfrac{4+6\theta}{10}\norm{\fkw}{V}.\notag
  \end{align}
That is, we can take~$\clV^{k,\rm swi}(t) = \clV^{4}(\fkw)(t)$.
\end{proof}

 %%%%%%%%%%%%%%%%%%%%%%%%
 %%%%%%%%%%%%%%%%%%%%%%%%
\subsection{Stabilizability by means of switching controls}\label{sS:asynull-switch}
 Let us now consider~\eqref{sys-w-fIk} in the infinite time-horizon as
 \begin{align}\label{sys-w-fI}
 &\dot w+Aw+A_{\rm rc}^A w=f,\qquad w(0)=\fkw_0\in V,\qquad t>0.
\end{align}

The next result shows that by concatenating the controls in Corollary~\ref{C:main-switch-sq}  we obtain a control driving the system asymptotically to zero.
\begin{theorem}\label{T:main-switch-asy}
Let us consider the concatenated control
 \begin{equation}\notag
u(t)\widehat\Phi^{\rm swi}(t)\coloneqq \clV^{k,\rm swi}(t),\quad\mbox{for}\quad t\in I_k^T,
 \end{equation}
 with~$\clV^{k,\rm swi}(t)$ given by Corollary~\ref{C:main-switch-sq}. Then,
the solution of~\eqref{sys-w-fI} with~$f=u\widehat\Phi^{\rm swi}$, satisfies
\begin{align}\notag
  \norm{w(t)}{V}&\le\ovlineC{T,\theta^{-1},\fkK,M,C_{\rm rc}}\rme^{-\mu(t-s)}\norm{w(s)}{V},\qquad t\ge s\ge0,
\end{align}
where~$\mu\coloneqq\frac1T\log(\tfrac{5}{2+3\theta})$.  Furthermore, ~$\norm{u(t)}{\bbR}\le \ovlineC{\fkK,M}\norm{\fkw_0}{V}$.
 \end{theorem}

 \begin{proof}
By concatenating the control forcings given by  Corollary~\ref{C:main-switch-sq} we arrive at a concatenated control~$u\widehat\Phi^{\rm swi}$ and, for the associated solution, we find that
  \begin{align}
 & \norm{w(kT+T)}{V}\le\tfrac{2+3\theta}{5}\norm{w(kT)}{V}\le(\tfrac{2+3\theta}{5})^{k+1}\norm{\fkw_0}{H},\notag\\
  \mbox{and}\quad&\norm{u(t)}{\bbR}\le M\fkK\norm{w(kT)}{V}\max_{1\le j\le M}\norm{\widehat\Phi_{j}}{H}\le M\fkK(\tfrac{2+3\theta}{5})^{k}\norm{\fkw_0}{V}.\notag
  \end{align}
Note that the magnitude of the control is bounded by~$M\fkK\norm{w(kT)}{V}$ in Corollary~\ref{C:main-switch-sq}.

Now, by choosing~$k\in\bbN$ such that
  $t\in I_k^T= (kT, kT+T]$ and recalling Lemma~\ref{L:strg-sol-w}, it follows for every~$t>0$ that 
\begin{align}\notag
\norm{w(t)}{V}\le D_Y^\frac12\left(\norm{w(kT)}{V}+\norm{u_\Phi}{L^2(I_k^T,H)}\right)
\le D_Y^\frac12(\tfrac{2+3\theta}{5})^{k}\left(1+T^\frac12M\fkK\right)\norm{\fkw_0}{V},
    \end{align}
    which implies
\begin{align}\notag
\norm{w(t)}{V}
\le D_Y^\frac12\left(1+T^\frac12M\fkK\right) D\rme^{-\mu t}\norm{\fkw_0}{V}
    \end{align}
 with~$D=\tfrac5{2+3\theta}$ and~$\mu=-\frac1T\log(\tfrac{2+3\theta}{5})$. Note that, since~$t\in I_k^T$, with~$\vartheta\coloneqq\tfrac{2+3\theta}{5}<1$ we have that $\vartheta^{k}< \vartheta^{\frac{t-T}T}=
 \vartheta^{-1}\vartheta^{\frac{1}Tt}$.
Thus, the result  follows for the concatenation of the inputs given in Corollary~\ref{C:main-switch-sq}.
 \end{proof}

\section{Proof of Theorem~\ref{T:main-Intro-swi}}\label{S:proofT:main-Intro}
I order to apply the results derived for the abstract evolutionary equations, we start by writing
~\eqref{sys-y-axy} as~\eqref{sys-y-Intro},
\begin{align}\label{sys-y-parab}
 \dot y+A y+A_{\rm rc}y=u
 \deltafun_{c},\quad y(0)=y_0,\quad t>0,
\end{align}
with~$A\coloneqq-\nu\Delta +\Id$  and~$A_{\rm rc}z=A_{\rm rc}(t)z\coloneqq(a(t,\Bigcdot)-1)z+b(t,\Bigcdot)\cdot\nabla z$.

%%%%%%%%%%%%%%%%%%%
%%%%%%%%%%%%%%%%%%%
\subsection{Satisfiability of Assumptions~\ref{A:A0sp}--\ref{A:A1}}\label{sS:satA0sp-A1}
The satisfiability of Assumptions~\ref{A:A0sp}--\ref{A:A1}, under~\eqref{assum.ab.parab}, is shown in~\cite[Sect.~5.1]{KunRodWal24-cocv}  where~$A\in\clL(V,V')$ is considered under the given boundary conditions (bcs). Namely,
with pivot space as~$H=L^2(\Omega)$ and the space~$V$  as in~\eqref{bcs.parab}.

 Note that, we can also rewrite~\eqref{sys-y-parab} above under normalized actuators~$\widehat\deltafun_{c}$ as in~\eqref{sys-y-Intro}, simply by rescaling~$u$ accordingly,
 \begin{align}\notag
  u \deltafun_{c}=\widehat u\widehat \deltafun_{c}
  \quad\mbox{with}\quad \widehat
 u=\norm{\deltafun_{c}}{\rmD(A)'}u \quad\mbox{and}\quad\widehat\deltafun_{c}=\norm{\deltafun_{c}}{\rmD(A)'}^{-1}\deltafun_{c}.
\end{align}

Therefore, it remains to check the satisfiability of Assumption~\ref{A:MstaticAct} and to specify the application of the abstract result to our concrete system with Dirac delta actuators. This is done in the following sections.

%%%%%%%%%%%%%%%%%%%
%%%%%%%%%%%%%%%%%%%
\subsection{Satisfiability of Assumption~\ref{A:MstaticAct}}\label{sS:satAMstaticAct}
Recalling Remark~\ref{R:muTtheta}, to show the complete stabilizability result in Theorem~\ref{T:main-Intro-swi}, we need a stronger result than the one in Assumption~\ref{A:MstaticAct}, in order to allow us to impose the rate~$\mu_{T,\theta}\coloneqq\frac1T\log(\tfrac{5}{2+3\theta})>0$ as in Theorem~\ref{T:main-switch-asy}. 

Let us fix an arbitrary desired rate~$\mu>0$. Let us also fix~$\theta=\theta_0\in(0,1)$ and set~$T_0\coloneqq\mu^{-1}\log(\tfrac{5}{2+3\theta_0})$. We show that Assumption~\ref{A:MstaticAct} holds with~$(T,\theta)=(T_0,\theta_0)$, which gives us the desired exponential decrease rate~$\mu_{T_0,\theta_0}=\frac1{T_0}\log(\tfrac{5}{2+3\theta_0})=\mu$ as in Theorem~\ref{T:main-switch-asy}.

We set~$\overline \mu\coloneqq\log(\theta_0^{-1})T_0^{-1}$. From~\cite[Thm.~3.1]{KunRodWal24-cocv} we know that
\begin{align*}
 \dot y+A y+A_{\rm rc}y=-\lambda\sum_{j=1}^{M}\langle\deltafun_{x^j},A^{-1}y\rangle_{\rmD(A)',\rmD(A)}\deltafun_{x^j},\quad y(0)=y_0\in V',
\end{align*}
is exponentially stable with rate~$\overline\mu$ in the~$V'$-norm, for~$\lambda>0$ and~$M$   large enough, where the delta actuators are located at appropriate positions in the spatial domain, $x^j\in\bfx\subset\Omega\subset\bbR^d$, $d\in\{1,2,3\}$; see~\cite[Sect.~5.2]{KunRodWal24-cocv}.
Furthermore, the~$V'$-norm will be strictly decreasing. In other words, the solution will satisfy
\begin{align}\label{sys-y-parab-obli-exp}
 \norm{y(t)}{ V'}\le \ex^{-\overline \mu (t-s)}\norm{y(s)}{ V'}, \qquad t\ge s\ge0.
\end{align}
In particular, we find that for all~$k\in\bbN$,
\begin{align}
&\norm{y(kT_0+T_0)}{ V'}\le \ex^{-\overline \mu T_0}\norm{y(kT_0)}{ V'}=\theta_0\norm{y(kT_0)}{ V'},\label{squeez-T0theta0}\\
&\mbox{with}\quad \theta_0\in(0,1)\mbox{ and }T_0\coloneqq\mu^{-1}\log(\tfrac{5}{2+3\theta_0})\notag
\end{align}
which gives us~\eqref{goal-static} with~$(T,\theta)=(T_0,\theta_0)$ as claimed above.

It remains to show the satisfiability of points~\ref{A:MstaticAct-norm}--\ref{A:MstaticAct-linfu} in Assumption~\ref{A:MstaticAct}. We can simply normalize the delta distributions~$\Phi_j\coloneqq\norm{\deltafun_{x^j}}{\rmD(A)'}^{-1}\deltafun_{x^j}$ to obtain the satisfiability of ~\ref{A:MstaticAct-norm}. Then it remains to check ~\ref{A:MstaticAct-linfu} concerning the control input
\begin{align}\label{vk-check}
v(t)\in\bbR^M\quad\mbox{with}\quad v_j(t)\coloneqq -\lambda\norm{\deltafun_{x^j}}{\rmD(A)'}\langle\deltafun_{x^j},A^{-1}y(t)\rangle_{\rmD(A)',\rmD(A)}.
\end{align}

Recalling the notation~$A_{\rm rc}^A=A^{-1}A_{\rm rc}A\in\clL(\rmD(A),V)$, for~$w\coloneqq A^{-1}y$ we find
\begin{align*}
 \dot w+A w+A_{\rm rc}^Aw=-\lambda\sum_{j=1}^{M}v_j A^{-1}\Phi_j,
\end{align*}
with~$w(0)=A^{-1}y_0\in V$, and we  have
\begin{align}\label{sys-y-parab-obli-exp-st}
 \norm{w(t)}{V}\le \ex^{-\overline \mu (t-s)}\norm{w(s)}{V}, \qquad t\ge s\ge0.
\end{align}

By arguments as in~\eqref{strg-sol-dw} we find that
\begin{align}
\tfrac{\rmd}{\rmd t}\norm{w}{V}^2&= -2\norm{w}{\rmD(A)}^2-2(A_{\rm rc}^A w,Aw)_H-2\lambda\sum_{j=1}^{M}\left(\langle\deltafun_{x^j},w\rangle_{\rmD(A)',\rmD(A)}\right)^2\notag\\
&= -2\norm{w}{\rmD(A)}^2-2(A_{\rm rc}^A w,Aw)_H
\le -\norm{w}{\rmD(A)}^2+C_{\rm rc}^2\norm{w}{V}^2.\label{sys-y-parab-delta-stabA-es}
\end{align}

After time integration of~\eqref{sys-y-parab-delta-stabA-es}, it follows that, for all~$t>s\ge0$,
\begin{align}
&\norm{w(t)}{V}^2+\norm{w}{L^2((s,t),\rmD(A))}^2\le\norm{w(s)}{V}^2+C_{\rm rc}^2\norm{w}{L^2((s,t),V)}^2.\label{sys-w-LinfHL2V}
\end{align}
In particular, using~\eqref{sys-y-parab-obli-exp-st},
\begin{align}
&\norm{w}{L^2((s,+\infty),\rmD(A))}^2\le\widehat C\norm{w(s)}{V}^2,\label{sys-y-L2int}
\end{align}
with~$\widehat C\coloneqq 1+C_{\rm rc}^2\tfrac{1}{2\mu}$.
By~\eqref{sys-y-L2int}, the input control~$v^*=(v^*_1,v^*_2,\dots,v^*_{M})$ with
\begin{align}
v^*_j&\coloneqq-\lambda\langle\deltafun_{x^j},w\rangle_{\rmD(A)',\rmD(A)}\notag
\intertext{satisfies~$v^*\in L^2((s,+\infty),\bbR^{M})$ and, with~$\dnorm{\deltafun}{}\coloneqq\max_{1\le j\le M}\norm{\deltafun_{x^j}}{\rmD(A)'}$,}
\norm{v^*}{L^2((s,+\infty),\bbR^{M})}^2&\le\lambda^2\dnorm{\deltafun}{}^2\norm{w}{L^2((s,+\infty),\rmD(A))}^2\le\lambda^2\dnorm{\deltafun}{}^2\widehat C\norm{w(s)}{V}^2.\label{inputL2}
\end{align}
This square integrability bound is not enough for ~\ref{A:MstaticAct-linfu} in Assumption~\ref{A:MstaticAct}, where the input is required to be essentially bounded.
Next, to obtain such boundedness we will use tools from optimal control. For this purpose, it is convenient to work with initial states in the pivot space~$H$. Thus, we consider ~$z\coloneqq A^{-\frac12}y=A^{\frac12}w$, which satisfies
\begin{align}\label{sys-z-opt}
 \dot z+A z+Z_{\rm rc}z=-\lambda\sum_{j=1}^{M}\langle\deltafun_{x^j},A^{-\frac12}z\rangle_{\rmD(A)',\rmD(A)} A^{-\frac12}\deltafun_{x^j},
\end{align}
with~$Z_{\rm rc}\coloneqq A^{-\frac12}A_{\rm rc}A^{\frac12}$ and initial state~$z(0)=A^{-\frac12}y_0\in H$. Due to~\eqref{sys-w-LinfHL2V}, we find
\begin{align}
&\norm{z(t)}{H}^2+\norm{z}{L^2((s,t),V)}^2\le\norm{z(s)}{H}^2+C_{\rm rc}^2\norm{z}{L^2((s,t),H)}^2.\label{sys-z-LinfHL2V}
\end{align}
Recalling~\eqref{inputL2}, we have that~$\tfrac12\norm{(z,v^*)}{L^2(\bbR_s,H\times\bbR^M)}^2$ is bounded. This suggests to look for a control~$\overline v$ and corresponding state~$\overline z$ minimizing the functional
\begin{subequations}\label{OCPs}
\begin{align}
&\clJ_s(z,v)\coloneqq\tfrac12\norm{(z,v)}{L^2(\bbR_s,H\times\bbR^M)}^2
\intertext{For a given generic initial state~$\fkh\in H$, we consider the minimizer~$(\overline z,\overline v)$,}
&\clJ_s(\overline z,\overline v)=\min_{(z,v)\in\clX_\fkh}\clJ_s(z,v),
\intertext{in the set}
& \clX_\fkh\coloneqq\left\{(z,v)\in \clX^0\mid
\dot z+A z+Z_{\rm rc}z= Bv\mbox{ and }z(s)=\fkh\right\},\\
\intertext{with}
&\clX^0\coloneqq W(\bbR_s,V,V')\times L^2(\bbR_s,\bbR^M),\\
& B\in\clL(\bbR^M,V'),\quad v\mapsto Bv\coloneqq \sum_{j=1}^{M}v_j A^{-\frac12}\deltafun_{x^j}.
\end{align}
\end{subequations}

Note that this is a linear convex optimization problem, for which we have one, and only one, minimizer~$(\overline z,\overline v)=(\overline z,\overline v)(\fkh)$, for each~$\fkh\in H$.
Recalling~\eqref{sys-y-parab-obli-exp-st} and~\eqref{inputL2}, by minimality we have that the optimal cost satisfies~$\clJ_s(\overline z,\overline v)\le C_J\norm{\fkh}{H}^2$ with~$ C_J\coloneqq\frac12\max\{\frac{1}{2\mu},\lambda^2\dnorm{\deltafun}{}^2\widehat C\}$. It is also well known that it can be written as~$\tfrac12(\Pi(s) \fkh,\fkh)_H$, where~$\Pi(s)\in\clL(H)$ is a nonnegative symmetric operator solving a differential Riccati equation in the time interval~$(s,+\infty)$. Combining first order optimality conditions and the dynamical programming principle, we find that the optimal control is given by
\begin{equation}\label{opt-u}
\overline u(t)=-B^*\Pi(t)\overline z(t)=-B^*p(t),
\end{equation}
 where~$p= \Pi\overline z$ is the adjoint state, which satisfies
\begin{equation}\label{p.allHaaV}
p(t)\in H\mbox{ for all }t\ge s,\mbox{ and } p(t)\in V\mbox{ for almost all }t\ge s,
\end{equation}
and also the dynamics
\begin{equation}\label{p.dyn}
\dot p=Ap+Z_{\rm rc}^*p-\overline z, \qquad\mbox{for all } t\ge s.
\end{equation}
which is a parabolic-like equation, when seen backwards in time.

In order to show the essential boundedness of the control input required in Assumption~\ref{A:MstaticAct} we will show that we have the improved analogue of~\eqref{p.allHaaV} as follows,
\begin{equation}\label{p.allVaaDA}
p(t)\in V\mbox{ for all }t\ge s,\mbox{ and } p(t)\in \rmD(A)\mbox{ for almost all }t\ge s.
\end{equation}
For that we have a closer look at the reaction-convection term~$A_{\rm rc}h= ah+b\cdot\nabla h$, with~$(a,b)$ satisfying~\eqref{assum.ab.parab}.
Observe that, by writing
\begin{equation}\label{def_Arc}
\langle A_{\rm rc}h,\phi\rangle_{\rmD(A)',\rmD(A)}\coloneqq \langle h, a\phi+\nabla\cdot(b\phi)\rangle_{V',V},
\end{equation}
 we can see that~$A_{\rm rc}(t)\in\clL(V',\rmD(A)')$ (for a.e.~$t>0$). Note also that the definition in~\eqref{def_Arc} holds for smooth tuples~$(h, a,b,\phi)$, as a consequence of direct computations using integration by parts; then, we can indeed write~\eqref{def_Arc}  due to a density argument.

Therefore,~$Z_{\rm rc}=A^{-\frac12}A_{\rm rc}A^{\frac12}\in\clL(H,V')$ and consequently~$Z_{\rm rc}^*\in\clL(V,H)$. The latter allows us to conclude~\eqref{p.allVaaDA}, from standard energy estimates for strong solutions. Furthermore, we have the smoothing property of parabolic equations  as follows (cf. arguments in~\cite[Sect.~4.2]{KunRod23-dcds}),
\begin{align}
\norm{p(t)}{V}^2&\le C_1\left(\norm{p(t+1)}{H}^2+\norm{\overline w}{L^2((t,t+1),H)}^2+\norm{p}{L^2((t,t+1),H}^2\right)\notag\\
&\le C_2\left(\norm{\overline z(t+1)}{H}^2+\norm{\overline z}{L^2((t,t+1),H)}^2\right)\le C_3\norm{\overline z(t)}{H}^2,\qquad t\ge s,\notag
\end{align}
from which we derive, for the optimal input~\eqref{opt-u},
 \begin{align}
\norm{\overline v(t)}{\ell^2}^2= \norm{B^* p(t)}{\ell^2}^2\le\norm{B^* }{\clL(V,\bbR^M)}^2C_3\norm{\overline z(t)}{H}^2.\label{uLinf1}
 \end{align}

Hence the optimal state satisfies
\begin{equation}\notag
 \dot{\overline z}+A\overline  z+Z_{\rm rc}\overline  z+BB^*\Pi\overline  z=0,
\end{equation}
which combined with the cost boundedness~$\clJ_s(\overline z,\overline v)\le C_J\norm{\fkh}{H}^2$ and with Datko's Theorem\cite[ Thm.~1]{Datko72}, leads us to
\begin{equation}\label{eD-Datko}
\norm{\overline z(t)}{H}^2\le D\rme^{-\varepsilon(t-s)}\norm{\overline z(s)}{H}^2,
 \end{equation}
 for suitable constants~$\varepsilon>0$ and~$D\ge1$. By~\eqref{uLinf1},
the essential boundedness of the control input  follows as
  \begin{align}
 \norm{\overline v}{L^\infty((s,s+T),\bbR^M)}^2\le \fkK\norm{\overline z(s)}{H}^2=\fkK\norm{A^{\frac12}\overline z(s)}{V'}^2,\quad\mbox{for all } s\ge0,\mbox{ } T>0,\notag
 \end{align}
with~$\fkK=\norm{B}{\clL(\bbR^M,V')}^2C_3D$.
Now, for~$\overline y\coloneqq A^{\frac12}\overline z$ we find
\begin{align}
& \dot{\overline y}+A \overline y+A_{\rm rc}\overline y=\sum_{j=1}^{M}\widehat  v_j \norm{\deltafun_{x^j}}{\rmD(A)'}^{-1}\deltafun_{x^j}\quad\mbox{with}\quad\widehat  v_j\coloneqq \norm{\deltafun_{x^j}}{\rmD(A)'}\overline v_j.\notag
\end{align}
Note that, the  control input~$\widehat  v$ satisfies
  \begin{align}
 \norm{\widehat  v}{L^\infty((s,s+T),\bbR^M)}^2\le \norm{\overline v}{L^\infty((s,s+T),\bbR^M)}^2\le \dnorm{\deltafun}{}^2\fkK\norm{\overline y(s)}{V'}^2,\quad\mbox{for all } s\ge0,\mbox{ } T>0,\notag
 \end{align}
and we conclude that point~\ref{A:MstaticAct-linfu} in Assumption~\ref{A:MstaticAct} is satisfied with~$v=\widehat  v$.

%%%%%%%%%%%%%%%%%%%
%%%%%%%%%%%%%%%%%%%
\subsection{Proof of Theorem~\ref{T:main-Intro-swi}}
 Let~$A$ and~$A_{\rm rc}$ be as in~\eqref{sys-y-parab} and let~$\mu>0$. We already know (from Sects~\ref{sS:satA0sp-A1} and~\ref{sS:satAMstaticAct}), that the Assumptions~\ref{A:A0sp}--\ref{A:MstaticAct} are satisfied for suitable normalized actuators~$\Phi_j=\norm{\deltafun_{x^j}}{\rmD(A)'}^{-1}\deltafun_{x^j}$, $1\le j\le M$. In particular, Assumption~\ref{A:MstaticAct} holds with~$(T,\theta)=(T_0,\theta_0)$ as in~\eqref{squeez-T0theta0}. By Theorem~\ref{T:main-switch-asy}, with~$M$ as in Corollary~\ref{C:main-switch-sq},  we can construct a switching control force~$u(t)A^{-1}\Phi^{\rm swi}(t)$, with~$\Phi^{\rm swi}(t)\in\deltafun_\bfx$ and such that~$w$ solves, for~$t>0$,
\begin{align}\label{sys-w-fI-app}
 &\dot w+Aw+A_{\rm rc}^A w=uA^{-1}\Phi^{\rm swi},\qquad w(0)=A^{-1}y_0\in V,
\end{align}
and satisfies~$\norm{w(t)}{V}\le\rme^{-\mu(t-s)}\norm{w(s)}{V}$, for~$t\ge s\ge0$. Therefore, $y=Aw$ solves, for~$t>0$,
\begin{align}\label{sys-y-fI-app}
 &\dot y+Ay+A_{\rm rc}y=u\Phi^{\rm swi},\qquad y(0)=y_0\in V',
\end{align}
and satisfies~$\norm{y(t)}{V'}\le\rme^{-\mu(t-s)}\norm{y(s)}{V'}$. Finally, note that, necessarily~$\Phi^{\rm swi}(t)=\deltafun_{c(t)}$ with a suitable piecewise constant function~$c(t)\in\bfx$. \qed

%%%%%%%%%%%%%%%%%%%%%%%%%%%%%%%%%%%%%%%
%%%%%%%%%%%%%%%%%%%%%%%%%%%%%%%%%%%%%%%
%%%%%%%%%%%%%%%%%%%%%%%%%%%%%%%%%%%%%%%
\section{Receding horizon control}\label{S:stabil_rhc}
It is a standard procedure to look for stabilizing controls as the minimizers of appropriate energy-like functionals. This is also important for applications where the minimization of the spent energy represented by those functionals is demanded.

%%%%%%%%%%%%%%%%%%%%%%%%%%%%%%%%%%%%%%%
%%%%%%%%%%%%%%%%%%%%%%%%%%%%%%%%%%%%%%%
\subsection{The framewok}\label{sS:rhc-frwork}
In order to simplify the exposition, we denote the solution of
\begin{align}\label{conrol_sys_b}
& \dot y+Ay+A_{\rm rc}(t)y=\sum^{M}_{j =1} \mathbf{u}_j(t) \deltafun_{x^j},  \qquad y(\overline t_0)=\overline y_0,\qquad \overline t_0<t<\overline t_0+T,
\intertext{for a given  tuple~$(\overline t_0,T,\overline y_0)\in \overline\bbR_{+}\times (\bbR_{+}\cup\{\infty\})\times V'$, and~$\bfu\in L^2((\overline t_0,\overline t_0+T),\bbR^M)$ by}
&\rmY_T(\overline t_0,\overline y_0;\bfu)(t)\coloneqq y(t),\quad\mbox{with}\quad\bfu(t)\coloneqq( \mathbf{u}_1(t),\dots, \mathbf{u}_M(t))\in\bbR^M.\notag
  \end{align}
 Then we fix~$\beta>0$ and define  the functional
\begin{equation}\notag
J_{T}(\mathbf{u}; \overline t_0,  \overline{y}_0) = \frac{1}{2} \int^{\overline t_0 +T}_{\overline t_0} \left(   \norm{\rmY_T(\overline t_0,\overline y_0;\bfu)(t)}{H}^2+\beta \norm{\mathbf{u}(t)}{\ell^2}^2\right)\rmd t.
\end{equation}
Finally, we denote the set of admissible controls as
\begin{equation}
 \mathcal{U}^{\overline t_0 ,T}_{\rm ad} := \{ \mathbf{u}\in  L^2((\overline t_0,\overline t_0 +T),\bbR^{M}) \mid \norm{\bfu(t)}{0}\leq 1,  \text{ for a.e.  } t \in (\overline t_0, \overline t_0 +T) \}.\notag
\end{equation}

For a given initial state~$y_0\in V'$, we would look for a stabilizing input as the solution of the following infinite horizon problem
\begin{align}
\label{tag:opinf}\stepcounter{equation}
 \tag*{{\rm (\theequation:\text{$P_{\infty}^{0,y_0}$})}}
\inf_{\bfu\in\clU^{0,\infty}_{\rm ad}} J_{\infty}(\mathbf{u};  0, y_0).
\end{align}

To address problem~\ref{tag:opinf}, we utilize a receding horizon framework, where the solution is obtained through a sequence of finite time-horizon optimal controls problems defined over the time intervals~$(\overline t_0,\overline t_0+T)$, namely,  problems as
\begin{align}
\label{tag:opfin}\stepcounter{equation}
 \tag*{{\rm (\theequation:\text{$P_{T}^{\overline t_0,\overline{y}_0}$})}}
\inf_{\bfu\in\clU^{\overline t_0,T}_{\rm ad}} J_{T}(\mathbf{u};  \overline t_0, \overline y_0).
\end{align}

\begin{remark}
We emphasize that~$\norm{\cdot}{0}$ is neither continuous nor convex and  the sets of admissible controls $\mathcal{U}^{\overline{t}_0,T}_{\rm ad}$ and $\mathcal{U}^{0,\infty}_{\rm ad}$  are not weakly compact. Consequently, the existence of minimizers remains unclear.  We,  thus,  make the assumption that all finite horizon problems  \ref{tag:opfin} within the receding horizon strategy  are well-posed (see Assum.~\ref{A:assum.well}).
\end{remark}

\begin{assumption}
\label{A:assum.well}
For every initial data $(\overline t_0,  \overline{y}_0) \in ( \overline\bbR_{+} ,V')$ and prediction horizon $T>0$,  the finite horizon problem~\ref{tag:opfin} admits a solution.  That is, 
\begin{align}
 J_{T}(\mathbf{u}^{\overline t_0,\overline{y}_0,*}_T;\overline t_0,\overline{y}_0):=\min_{\bfu\in\clU^{\overline t_0,T}_{\rm ad}} J_{T}(\mathbf{u};\overline t_0,\overline{y}_0).\notag
\end{align}
\end{assumption}

In Algorithm~\ref{RHA} we recall the receding horizon strategy.
\begin{algorithm}[htbp]
\caption{Receding horizon control, RHC($\delta,T$)}\label{RHA}
\begin{algorithmic}[1]
\REQUIRE{The sampling time $\delta$,  the prediction horizon $T\geq \delta$,  and the initial state $y_0$.}
\ENSURE{A stabilizing control input~$\mathbf{u}_{rh}\in\clU_{\rm ad}^{0,\infty}$.}
\STATE Set~$(\overline t_0,\overline{y}_0):=(0, y_0)$  and  $y_{rh}(0) =y_0 $;
\STATE\label{alg:Tinf} Set~$T_{\infty}=\infty$;
\WHILE{$\overline t_0<T_{\infty}$}
\STATE\label{alg:solve} ``Solve'' the  open-loop problem~\ref{tag:opfin} to find~$  \mathbf{u}^{\overline t_0,\overline{y}_0,*}_T  \in\clU_{\rm ad}^{\overline t_0,T}$;
\STATE For all $t \in [\overline t_0, \overline t_0+\delta)$,  set $\mathbf{u}_{rh}(t)= \mathbf{u}^{\overline t_0,\overline{y}_0,*}_T(t)$ and~$y_{rh}(t)=\rmY_T(\overline t_0,\overline{y}_0;\mathbf{u}^{\overline t_0,\overline{y}_0,*}_T)(t)$;
\STATE Update: $(\overline t_0,\overline{y}_0)  \leftarrow (\overline t_0+\delta,  \rmY_T(\overline t_0,\overline{y}_0;\mathbf{u}^{\overline t_0,\overline{y}_0,*}_T)(\overline t_0+\delta)$;
\ENDWHILE
\end{algorithmic}
\end{algorithm}

\begin{remark}
Note that  Step~\ref{alg:solve} in Algorithm \ref{RHA} is well-defined due to Assumption \ref{A:assum.well}.  Regarding Step~\ref{alg:Tinf} in Algorithm \ref{RHA},  in practice, for numerical simulations we will fix a finite number~$T_\infty$ for the computation time.
\end{remark}

%%%%%%%%%%%%%%%%%%%%%%%%%%%%%%%%%%%%%%%
%%%%%%%%%%%%%%%%%%%%%%%%%%%%%%%%%%%%%%%
\subsection{Stabilizability and suboptimality}\label{sS:rhc-stabsubopt}
We show that the receding horizon control provided by Algorithm \ref{RHA} is stabilizing. Further, we investigate the suboptimality of its associated cost to the optimal cost of the infinite-horizon problem.

Let us reconsider system~\eqref{sys-w-fIk}, now in the interval~$\clI_{\overline{t}_0}\coloneqq(\overline{t}_0, \overline{t}_0+T)$
\begin{align}\label{sys-w-fIdelta}
 &\dot w+Aw+A_{\rm rc}^A w=f,\qquad w(\overline{t}_0)=\overline\fkw,\qquad t\in \clI_{\overline{t}_0}.
\end{align}

We shall need the following auxiliary result.
\begin{lemma}
The solution of~\eqref{sys-w-fIdelta} satisfies
\begin{align}
\norm{w}{ L^2(\clI_{\overline{t}_0},\rmD(A))}^2   & \le D_1\left(\norm{\overline{\fkw}}{V}^2+\norm{w}{L^2(\clI_{\overline{t}_0},V)}^2+\norm{f}{L^2(\clI_{\overline{t}_0},H)}^2\right);\label{e13}\\
 \norm{\overline\fkw}{V}^2&\le D_2\norm{w}{L^2(\clI_{\overline{t}_0},\rmD(A))}^2+\norm{f}{L^2(\clI_{\overline{t}_0},H)}^2;\label{e14}
\end{align}
where  $D_1=D_1(C_{\rm rc})$ is independent of $T$, and $D_2 = D_2 (T,C_{\rm rc})$, with~$C_{\rm rc}$ as in Assumption~\ref{A:A1}.
\end{lemma}

\begin{proof}
Proceeding as in Lemma \ref{L:strg-sol-w} (see \eqref{estimate_BA}),  we have that~\eqref{e13} holds true.
To show~\eqref{e14}, we start by introducing~$\widehat w\coloneqq \frac{T+\overline{t}_0-t}{T}w$, for which we find that
\begin{align}\notag
 &\dot{\widehat w}+A{\widehat w}+A_{\rm rc}^A {\widehat w}=\tfrac{T+\overline{t}_0-\Bigcdot}{T}f-\tfrac{1}{T}w,\qquad \widehat w(\overline{t}_0)=\overline\fkw,\qquad \widehat w(\overline{t}_0+T)=0,
\end{align}
and by testing the dynamics with~$2A\widehat w$ gives
\begin{align}
 &\tfrac{\rmd}{\rmd t}\norm{\widehat w}{V}+2\norm{\widehat w}{\rmD(A)}+2( A_{\rm rc}^A {\widehat w}, {\widehat w})_H=2(\tfrac{T+\overline{t}_0-\Bigcdot}{T}f-\tfrac{1}{T}w, {\widehat w})_H,\qquad \widehat w(\overline{t}_0)=\overline\fkw. \notag
\end{align}
Subsequently, time integration leads to
\begin{align}
 &-\norm{\overline\fkw}{V}+2\norm{\widehat w}{L^2(\clI_{\overline{t}_0},\rmD(A))}+2( A_{\rm rc}^A {\widehat w}, {\widehat w})_{L^2(\clI_{\overline{t}_0},H)}=2(\tfrac{T+\overline{t}_0-\Bigcdot}{T}f-\tfrac{1}{T}w, {\widehat w})_{L^2(\clI_{\overline{t}_0},H)}.\notag
\end{align}
Using  the Young inequality, and the fact that~$\rmD(A)\xhookrightarrow{}V\xhookrightarrow{}H$, we can obtain
\begin{align}
 &\norm{\overline\fkw}{V}^2\le C_1\norm{\widehat w}{L^2(\clI_{\overline{t}_0},\rmD(A))}^2+\norm{f}{L^2(\clI_{\overline{t}_0},H)}^2 +\tfrac{1}{T}\norm{w}{L^2(\clI_{\overline{t}_0},H)}^2,\notag
\end{align}
which implies~\eqref{e14} with  $D_2:=C_1+\tfrac{1}{T}$.
\end{proof}

Now we are in the position to investigate the stabilizability and suboptimality of Algorithm \ref{RHA}.  Beforehand,  we need to  define the finite- and infinite-horizon value functions.
\begin{definition}
For any $y_0 \in V' $  the infinite horizon value function $V_{\infty}$ is defined by
\begin{equation*}
 V_{\infty}(y_0)\coloneqq \inf_{\bfu\in\clU^{0,\infty}_{\rm ad}} J_{\infty}(\mathbf{u};  0, y_0).
\end{equation*}
Similarly, for every initial vector  $(\overline{t}_0,  \overline{y}_0) \in \overline\bbR_+ \times V'$ and every prediction horizon $T>0$,  the finite horizon value function $V_{T}$ is defined by
\begin{equation*}
V_{T}(\overline{t}_0,\overline{y}_0) \coloneqq \min_{\bfu\in\clU^{\overline{t}_0,T}_{\rm ad}} J_{T}(\mathbf{u};\overline{t}_0,\overline{y}_0).
\end{equation*}
\end{definition}

\begin{theorem}
\label{subopth}
Suppose that the assumptions of Theorem \ref{T:main-Intro-swi} and Assumption \ref{A:assum.well} hold true. Then, for every given  sampling time $\delta>0$, there exist numbers $T^* > \delta$, and $\alpha \in (0,1)$,  such that for every fixed  prediction horizon $T \geq T^*$,  the receding horizon pair  $(y_{rh} ,  \mathbf{u}_{rh})$ obtained from Algorithm \ref{RHA} satisfies for every $y_0 \in V'$ the suboptimality inequality
\begin{equation}
\label{ed27}
\alpha V_{\infty}(y_0) \leq \alpha J_{\infty}(\mathbf{u}_{rh};0,y_0) \leq V_{\infty}(y_0)
\end{equation}
 and  the  exponential stability
\begin{equation}
\label{ed28}
|y_{rh}(t,\cdot )|^2_{V'} \leq C_{rh}e^{-\zeta t}|y_0|^2_{V'} \quad  \text{ for }  t\ge 0,
\end{equation}
where the positive numbers $\zeta$ and  $C_{rh}$  depend on $\alpha$, $\delta$, and $T$, but are independent of $y_0$.
\end{theorem}
\begin{proof}
The proof follows with similar arguments as in~\cite[Proof of Thm.~2.6]{AzmiKunisch19}.  Therefore, we omit the details and restrict ourselves to the verification of the sufficient conditions/properties  P1--P3 in  \cite[Thm.~2.6]{AzmiKunisch19} with respect to the~$V'$-norm in place of the $H$-norm.
\begin{enumerate}[noitemsep,topsep=5pt,parsep=5pt,partopsep=0pt,leftmargin=0em]
\renewcommand{\theenumi}{{P\sf \arabic{enumi}}} 
 \renewcommand{\labelenumi}{}
 
 \item\theenumi:\label{P1} {\em For every $(\overline{t}_0, \overline{y}_0) \in \overline\bbR_+\times V'$ and $T>0$,  every finite horizon optimal control problem of the form~\ref{tag:opfin} admits a solution: }

Property~\ref{P1}  follows by Assumption \ref{A:assum.well}.

 \item\theenumi:\label{P2} {\em For every~$T>0$, $V_T$ is globally decrescent with respect to the $V'$-norm. That is, there exists a continuous,  non-decreasing,  and bounded function $\gamma_2: \mathbb{R}_+ \to \mathbb{R}_+$ such that
\begin{equation}
\label{e1a}
V_T(\overline{t}_0, \overline{y}_0) \le  \gamma_2(T)|\overline{y}_0|^2_{V'}  \qquad  \text{ for all } (\overline{t}_0, \overline{y}_0) \in \overline\bbR_{+}\times V',
\end{equation}
with~$\gamma_2$ independent of~$(\overline{t}_0,\overline{y}_0)$.}

To address~\ref{P2},  we first show that for every given $(\overline{t}_0, \overline{y}_0)$,  there exists a stabilizing control~$\mathbf{u}(\overline{y}_0) \in \mathcal{U}^{ \overline{t}_0, \infty}_{ad}$.   We start by observing that by shifting time, $t\coloneqq s+\overline{t}_0$, and denoting~$\widehat g(s)\coloneqq g(s+\overline{t}_0)$ we find that~$y=y(t)$ solves
\begin{align}\label{sys-y-parab-Mpc}
\tfrac{\rmd}{\rmd t} { y}+A{  y}+A_{\rm rc}(t){y}= {u(t)}
 \deltafun_{{c}(t)},\quad y(\overline{t}_0)= \overline{y}_0,\quad t>\overline{t}_0,
\end{align}
if, and only if,~$\widehat y=\widehat y(s)$ solves
\begin{align}\label{sys-y-parab-Mpc-shifted}
\tfrac{\rmd}{\rmd s} {\widehat y}+A{ \widehat y}+\widehat A_{\rm rc}(s){ \widehat y}= { \widehat u}(s)
 \deltafun_{ \widehat{c}(s)},\quad \overline y(0)= \overline{y}_0,\quad s>0.
\end{align}

In Section~\ref{sS:satA0sp-A1} we have addressed the satisfiability of Assumption~\ref{A:A1} for the concrete reaction-convection operator~$A_{\rm rc}=a\Id+b\cdot\nabla$. The same assumption is still satisfied by~$\widehat A_{\rm rc}$, since~$\norm{\widehat A_{\rm rc}}{L^\infty(\bbR_+,\clL(H,V'))}\eqqcolon\widehat C_{\rm rc}\le C_{\rm rc}$. Therefore,  Theorem~\ref{T:main-Intro-swi} will hold true with the same set of actuators. As a consequence, it follows that there exist a magnitude control function~$ u\in L^2((\overline{t}_0,+\infty),\bbR)$, and a piecewise constant path $ c(t)\in\bfx$  for which the associated solution of~\eqref{sys-y-parab-Mpc} satisfies
 \begin{align}
 &\norm{y(t,\Bigcdot)}{V'}\le C\ex^{-\mu (t-\overline{t}_0)}\norm{\overline{y}_0}{V'}\quad\mbox{and}\quad \norm{u(t)}{L^2((\overline{t}_0,+\infty),\bbR)}\le C_u\norm{ \overline{y}_0}{V'}, \label{rhc-timeshift-yu}
\end{align}
 for all~$t\ge \overline{t}_0$, 
where the constants~$C$ and~$C_u$ can be taken as in Theorem \ref{T:main-Intro-swi}, thus independently of~$(\overline{t}_0,\overline{y}_0)$.

Now,  by defining the vector input~$\bfu(t) =  (\bfu_1(t) ,\dots,\bfu_M(t) )$ as
\begin{equation}
\bfu_j(t) =\begin{cases} u(t),  & \text{ if  }\; {c}(t)  =x^j; \\0    & \text{ otherwise};
\end{cases}\qquad \mbox{for}\quad t\ge \overline{t}_0,
\end{equation}
we   see that
\begin{equation}
\sum^{M}_{j =1} \mathbf{u}_i(t) \deltafun_{x^j}  =  u(t)\deltafun_{ \hat{c}(t)}  \quad \text{ for }   t \geq \overline{t}_0.
\end{equation}
We can conclude that  $ \mathbf{u} \in \mathcal{U}^{\overline{t}_0,\infty}_1$ with
\begin{equation}
\label{e2a}
 \norm{{\mathbf{u}}}{L^2((\overline{t}_0,\infty);\mathbb{R}^M)}^2 = \norm{u( \overline{y}_0)}{L^2((\overline{t}_0,+\infty),\bbR)}^2\le C^2_u\norm{ \overline{y}_0}{V'}^2.
\end{equation}

Recalling the notation~$A_{\rm rc}^A=A^{-1}A_{\rm rc}A\in\clL(\rmD(A),V)$, for~$w\coloneqq A^{-1}y$ satisfying \eqref{sys-w-fIdelta} with  $\overline{\fkw}\coloneqq A^{-1}\overline{y}_0$ and $f\coloneqq \sum^{M}_{j =1} \mathbf{u}_jA^{-1} \deltafun_{x^j} $,  and using \eqref{e13} and \eqref{e2a},  we can write for \eqref{conrol_sys_b} that 
\begin{equation*}
\begin{split}
\norm{y}{ L^2(\clI_{\overline{t}_0},H)}^2  & \leq  D_1\left(|y(\overline{t}_0)|^2_{V'} + |y|^2_{L^2(\clI_{\overline{t}_0};V')} + \norm{\sum^{M}_{j =1} \mathbf{u}_j \deltafun_{x^j}}{L^2(\clI_{\overline{t}_0};D(A)')}^2 \right)\\
&\leq D_1\left( |\overline{y}_0|^2_{V'} + |y|^2_{L^2(\clI_{\overline{t}_0};V')} +M  \max_{1\leq  j \leq M  } \norm{\deltafun_{x^j}}{D(A)'}^2 \norm{\mathbf{u}}{L^2(\clI_{\overline{t}_0};\mathbb{R}^M)}^2  \right)\\
&\le D_1D_3(T)\norm{ \overline{y}_0}{V'}^2,\qquad\mbox{with}\quad D_3(T)\coloneqq \left(1+\tfrac{C^2(1-e^{-2\mu T})}{2\mu}+M\max_{1\leq  j \leq M  } \norm{\deltafun_{x^j}}{D(A)'}^2C_u^2\right).
\end{split}
\end{equation*}
Hence, necessarily the optimal (minimal) cost satisfies
\begin{equation*}
\begin{split}
V_T(\overline{t}_0, \overline{y}_0) &\leq J_{T}(\mathbf{u};\overline{t}_0, \overline{y}_0) =\tfrac12 \norm{y}{L^2(\clI_{\overline{t}_0}; H)}^2 +\tfrac \beta2 \norm{\mathbf{u}}{L^2(\clI_{\overline{t}_0}; \mathbb{R}^M)}^2 \\& \le  \tfrac12\left(D_1D_3(T)+  \beta C^2_{u} \right)\norm{ \overline{y}_0}{V'}^2=: \gamma_2(T) \norm{ \overline{y}_0}{V'}^2,
\end{split}
\end{equation*}
which implies ~\eqref{e1a}.

\item\theenumi:\label{P3}
 {\em For every $T>0$,  $V_T$ is uniformly positive with respect to the $V'$-norm. In other words,  for every $T>0$ there exists a constant $\gamma_1(T)>0$, independent of~$(\overline{t}_0, \overline{y}_0)$, such that 
\begin{equation}
\label{e8}
V_T(\overline{t}_0, \overline{y}_0) \geq  \gamma_1(T) \norm{\overline{y}_0}{V'}^2  \qquad  \text{ for all } (\overline{t}_0, \overline{y}_0) \in  \overline\bbR_+\times V'
\end{equation}
}

To verify property~\ref{P3}, we use again the fact that~$w\coloneqq A^{-1}y$ satisfies~\eqref{sys-w-fIdelta}  with  $\overline{\fkw}\coloneqq A^{-1}\overline{y}_0$ and $f\coloneqq \sum^{M}_{j =1}  \mathbf{u}_jA^{-1} \deltafun_{x^j} $ for an arbitrary $\mathbf{u} = (\mathbf{u}_1, \dots, \mathbf{u}_M) \in L^2(\clI_{\overline{t}_0}, \mathbb{R}^{M})$.  Then by~\eqref{e14}  we find
\begin{align}\label{e3}
&\norm{\overline{y}_0}{V'}^2  \leq  D_2\norm{y}{ L^2(\clI_{\overline{t}_0}, H)}^2+\norm{\sum^{M}_{j =1} \mathbf{u}_j(t) \deltafun_{x^j}}{ L^2(\clI_{\overline{t}_0},D(A)')}^2\\
& \leq  D_2\norm{y}{ L^2(\clI_{\overline{t}_0}, H)}^2+M  \max_{1\le j \le M  } \norm{\deltafun_{x^j}}{\rmD(A)'}^2 \norm{\bfu}{ L^2(\clI_{\overline{t}_0},  \mathbb{R}^M  )}^2\le \gamma_1(T)J_{T}(\bfu;\overline{t}_0,\overline{y}_0),\\
&\mbox{with}\quad\gamma_1(T) \coloneq \frac12 \left( \max\left\{ D_2,  \tfrac{M  \max\limits_{1\leq  j \leq M  } \norm{\deltafun_{x^j}}{\rmD(A)'}^2}{\beta}\right \} \right)^{-1}.
\end{align}
In particular, by taking a minimizer~$\mathbf{u}^{k,\overline{y}_0,*}_T$ as in Assumption~\ref{A:assum.well}, we obtain that
\begin{equation*}
\gamma_1(T) \norm{\overline{y}_0}{V'}^2 \leq J_T(\mathbf{u}^{k,\overline{y}_0,*}_T;\overline{t}_0, \overline{y}_0)=V_T(\overline{t}_0, \overline{y}_0),
\end{equation*}
which gives us~\eqref{e8}.
\end{enumerate}
\end{proof}

%%%%%%%%%%%%%%%%%%%%%%%%%%
%%%%%%%%%%%%%%%%%%%%%%%%%%
%%%%%%%%%%%%%%%%%%%%%%%%%%
\section{Numerical simulations}\label{S:num_impl}
To illustrate the theoretical findings we present  numerical experiments  utilizing  Algorithm~\ref{RHA} applied to  an unstable parabolic equation.   Specifically,   we address the infinite-horizon problem~\ref{tag:opinf} on the spatial domain  $\Omega := (0,1)^2 \subset \mathbb{R}^2$ with homogeneous Neumann boundary conditions,  that is,  $\clG=\clG_{\rm Neu}=\bfn\cdot\nabla$. 

For all numerical experiments, we employed a conforming linear finite element scheme utilizing continuous piecewise linear basis functions over a uniform triangulation for spatial discretization as depicted in Fig.~\ref{Fig:meshAct},
 \begin{figure}[htbp]
    \centering
 \subfigure%[]
    {\includegraphics[height=0.4\textwidth,width=0.4\textwidth]{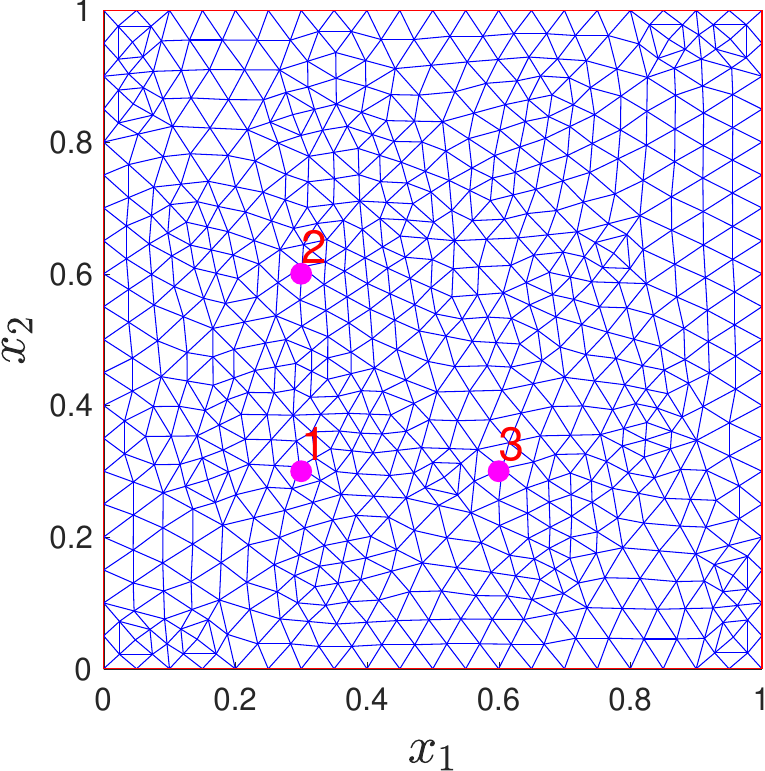}}
    \qquad
    \subfigure%[]
    {\includegraphics[height=0.4\textwidth,width=0.4\textwidth]{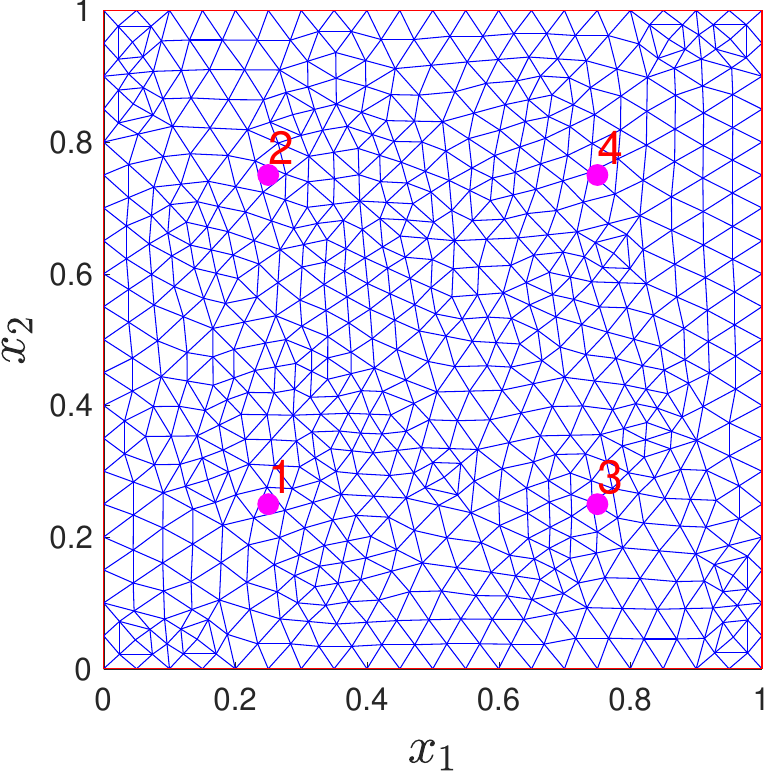}}\\
    \subfigure%[]
    {\includegraphics[height=0.4\textwidth,width=0.4\textwidth]{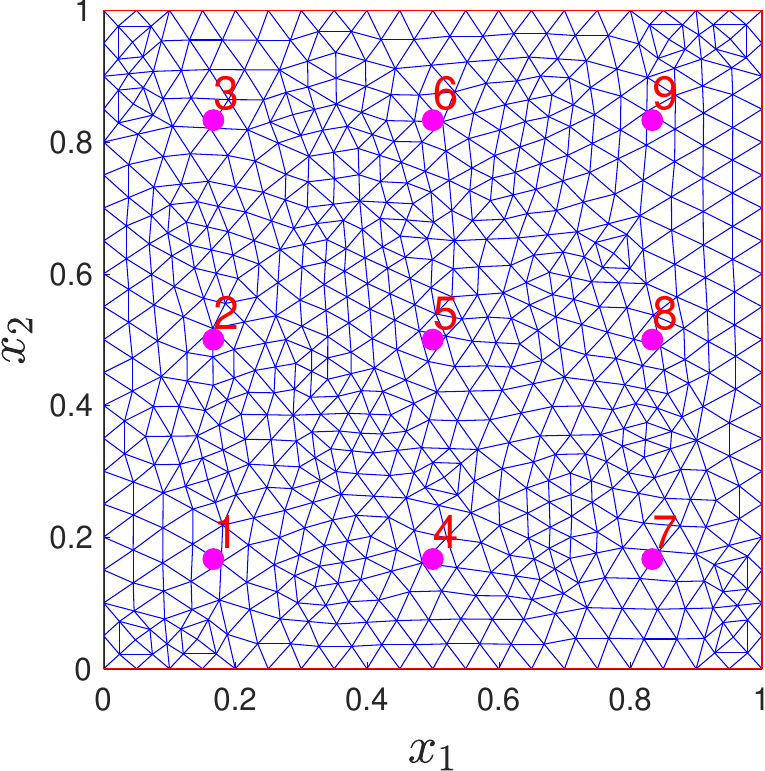}}
    \qquad
    \subfigure%[]
    {\includegraphics[height=0.4\textwidth,width=0.4\textwidth]{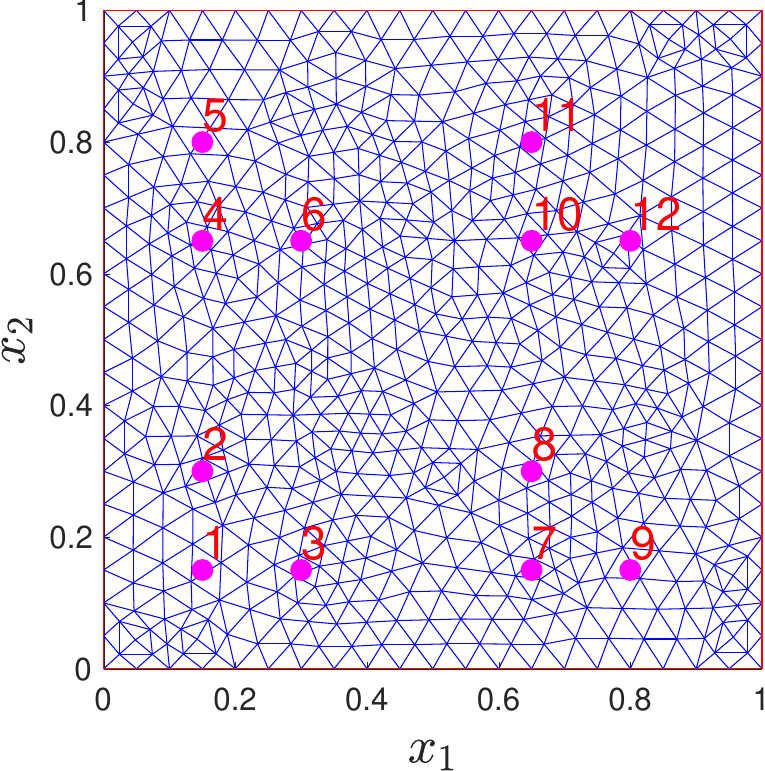}}
     \caption{Placement of the $M\in\{3,4,9,12\}$ Dirac actuators.}
     \label{Fig:meshAct}
\end{figure}
where we also show the four scenarios for the locations of the point actuators we shall consider in the simulations.

 The resulting ordinary differential equations from spatial discretization were solved numerically using the Crank-Nicolson time-stepping method with step-size of $\Delta t = 5 \times 10^{-3}$.  Additionally,   we set $\nu = 0.1$,  $\beta = 5 \times 10^{-4} $ and for $x := (x_1,x_2) \in \mathbb{R}^2$, we chose
\begin{equation}\notag
 a(t,x):=-2 + (2-x_1)\cos(\pi x_2) -0.2|\sin(t+x_2)|, \quad b(t,x):= \binom{\frac{t+2}{t+1}\left(x_1(x_1-1)x_2 \right) }{ -(x_1-0.5)x_2(x_2-1)\cos(t)},
\end{equation}
and $y_0(x):=x_1(1+\sin(2x_2))$. 

For all the experiments, in Algorithm~\ref{RHA} we have set $T = 1$ as the prediction horizon and $\delta = 0.25$ as the sampling time.  

To simplify the exposition, let us fix~$(\overline{t}_0,\overline{y}_0)$ and write~\ref{tag:opfin} as
\begin{align}\notag
&\clF(\mathbf{u}^*)\quad\longrightarrow\quad\min_{\bfu\in\clU^{\overline{t}_0,T}_{\rm ad}} \clF(\bfu),\qquad\mbox{with}\quad\mathcal{F}(\mathbf{u}) \coloneqq J_{T}(\mathbf{u};   \overline{t}_0, \overline{y}_0).
\end{align}
To tackle the problem,  we employed the proximal gradient method incorporated with the non-monotone line search, as outlined in \cite{AzmiBern2023}, leading to the iteration
  \begin{equation}\label{proj_Ik}
\mathbf{u}_{k + 1} \in  \proj_{\mathcal{U}^{\overline{t}_0,T}_{\rm ad}} \left(\mathbf{u}_k - \frac{1}{\alpha_k}\nabla \mathcal{F}(\mathbf{u}_k)\right).
\end{equation}
Here, $\proj_{\mathcal{U}_{\rm ad}^{\overline{t}_0, T}}$ denotes the orthogonal projection onto~$\mathcal{U}_{\rm ad}^{\overline{t}_0, T}$ defined pointwise in time as
\begin{equation}\label{proj0}
\mathbf{u}_{k + 1}(t) \in  \proj_{\mathcal{\overline{U}}_{\rm ad}} \left(\mathbf{u}_k(t) - \frac{1}{\alpha_k}\nabla \mathcal{F}(\mathbf{u}_k)(t)\right),
\end{equation}
where  $\mathcal{\overline{U}}_{\rm ad}\coloneqq \{  \mathbf{x} \in \mathbb{R}^{M} \mid\;  \norm{ \mathbf{x} }{0} \leq 1 \}$  and  and $\proj_{\mathcal{\overline{U}}_{\rm ad}}\colon\bbR^M\to\bbR^M$
is defined as follows
\begin{equation}\label{rhc-u-vec}
w=\proj_{\mathcal{\overline{U}}_{\rm ad}}v,\qquad w_j\coloneqq\begin{cases}
v_j,\quad&\mbox{if } j=j_v\coloneqq\min\{  i \in \{1, \dots,M\} \mid  \norm{v_i}{\bbR}=\norm{v}{\ell^\infty} \};\\
0,\quad&\mbox{otherwise };
\end{cases}
\end{equation}
where~$\norm{v}{\ell^\infty}\coloneqq\max\limits_{1\le i\le M}\norm{v_i}{\bbR}$. Furthermore,  to accelerate the proximal gradient method, we utilized the Barzial Borwein (BB) stepsizes \cite{BB}. Specifically, the stepsize $\alpha_k$ was computed by a non-monotone line search algorithm \cite{AzmiBern2023}, which employs the BB stepsize corresponding to the smooth part $\mathcal{F}$ as the initial trial stepsize. See \cite{AzmiKunisch22, AzmiKunisch20} for more details. The optimization algorithm was terminated based on the following condition
\begin{equation*}
\alpha_k \norm{\mathbf{u}_{k + 1}- \mathbf{u}_{k}}{L^2(\clI_{\overline{t}_0},\mathbb{R}^{M})} \leq 10^{-5}
\end{equation*}
All computations were carried out on the MATLAB platform.  Further,  for every  RHC $\mathbf{u}_{rh}$ obtained  by Algorithm \ref{RHA},  we find the corresponding switching control by
\begin{equation}\label{rhc-uc}
u_{rh}(t)=\sign({\mathbf{u}_{rh}(t)}_{j_{\mathbf{u}_{rh}(t)}})\norm{\mathbf{u}_{rh}(t)}{\ell^\infty}\quad\mbox{and}\quad c(t)=\delta_{x^{j_{\mathbf{u}_{rh}(t)}}},
\end{equation}
where  $j_{\mathbf{u}_{rh}(t)} \in \{1, \dots, M\} $ stands for the index corresponding the nonzero component of $\mathbf{u}_{rh}(t)$, as defined in \eqref{rhc-u-vec}.    
Based on the number and placement of the point actuators, we provide the following two examples.

%%%%%%%%%%%%%%%%%%%%%%%
%%%%%%%%%%%%%%%%%%%%%%%
\subsection{On the number and location of the actuators}
\label{exp1}
First of all,  we observe that free dynamics is unstable, as  illustrated in Fig.~\ref{Fig:free-dyn}, where the norm of the uncontrolled state is exponentially increasing.
\begin{figure}[htbp]
    \centering
    \subfigure[Free dynamics.\label{Fig:free-dyn}]
         {\includegraphics[height=0.35\textwidth,width=0.49\textwidth]{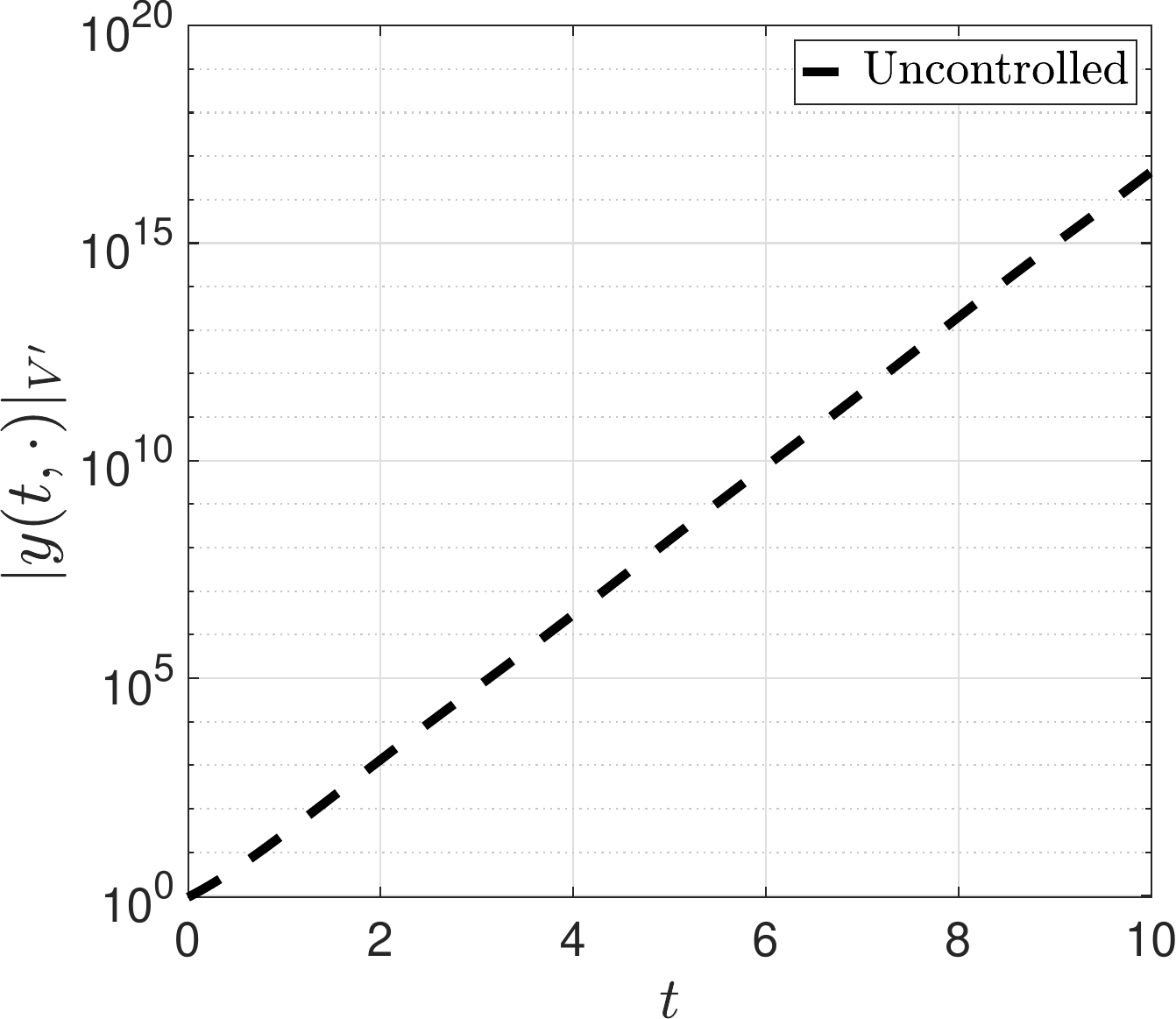}}
   \subfigure[Controlled dynamics for several~$M$. \label{Fig:stabil_all}]
{\includegraphics[height=0.35\textwidth,width=0.49\textwidth]{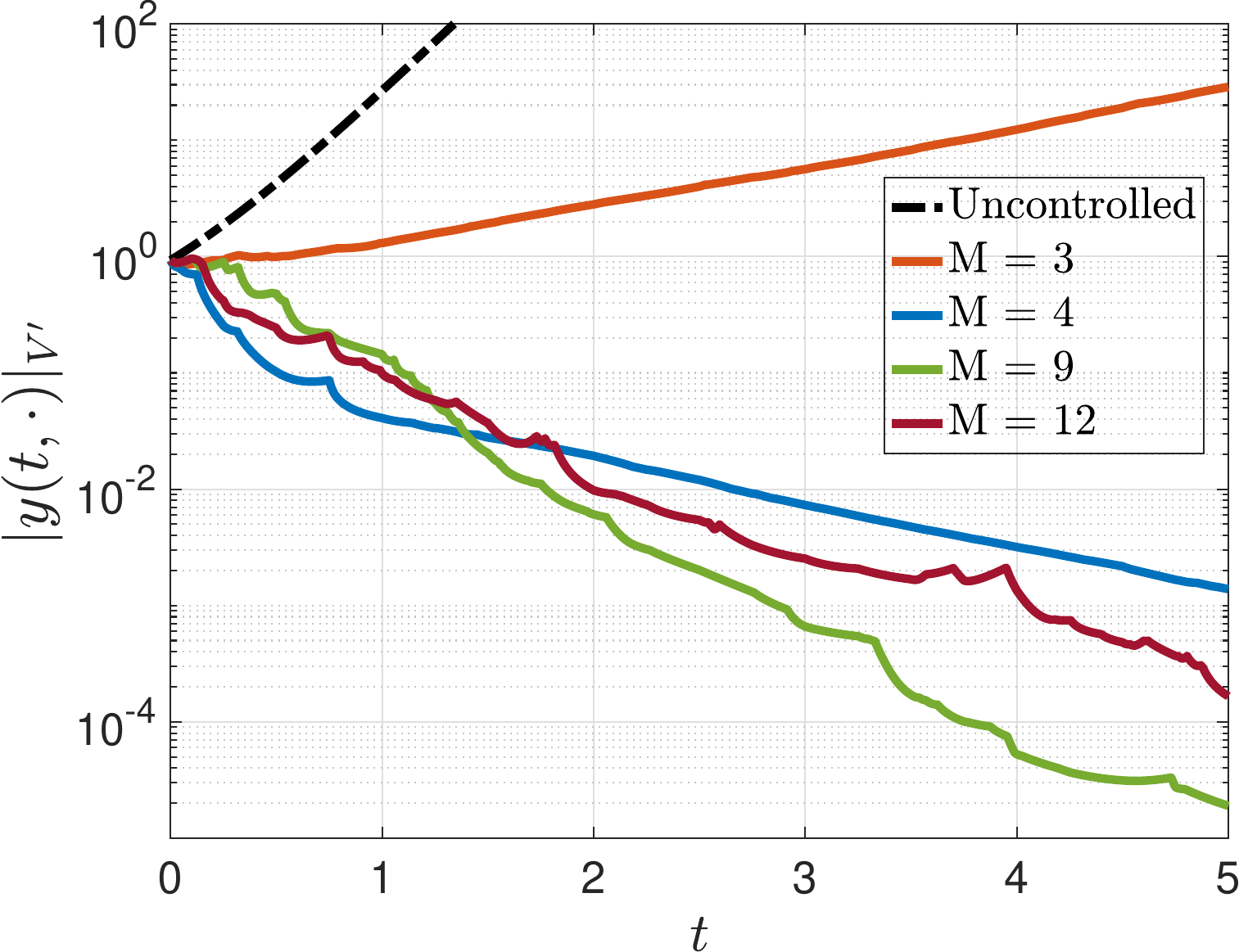}}
 \caption{Evolution of the $V'$-norm}
 \label{Fig:norm_evol}
\end{figure}
Next, we consider the cases of~$M\in\{3,4,9,12\}$ point actuators placed as in Fig.~\ref{Fig:meshAct}. The cases~$M\in\{3,12\}$ are inspired by the sequence of families of actuators leading to the nonswitching stabilizability result in~\cite{KunRodWal24-cocv} (see~\cite[Fig.~1]{KunRodWal24-cocv}) that we used to derive the switching stabilizability result in Theorem~\ref{T:main-Intro-swi}. The cases~$M\in\{4,9\}$ are rather inspired in the sequence proposed, for example, in~\cite[Fig.~2]{AzmiKunRod23}, which guarantees the nonswitching stabilizability result in the case the actuators are/were indicator functions.
In~Fig.~\ref{Fig:stabil_all},  we see that the switching receding horizon control (RHC) is not stabilizing in case~$M=3$, whereas it is in cases~$M\in\{4,9,12\}$.   Fig.~\ref{Fig:swicontrol3-12} depicts the switching pattern~$c(t)$ of the corresponding controls,  recall that~$c(t)=\deltafun_{x^j}$ if~$\bfu_j(t)\ne0$; see~\eqref{rhc-u-vec} and~\eqref{rhc-uc}. Note that, Fig.~\ref{Fig:swicontrol3-12} confirms that, in all cases~$M\in\{3,4,9,12\}$, no more than~$1$ Dirac is active at any given time instant.
 \begin{figure}[htbp]
    \centering
   \subfigure%[]
    {\includegraphics[height=.3\textwidth,width=.9\textwidth]{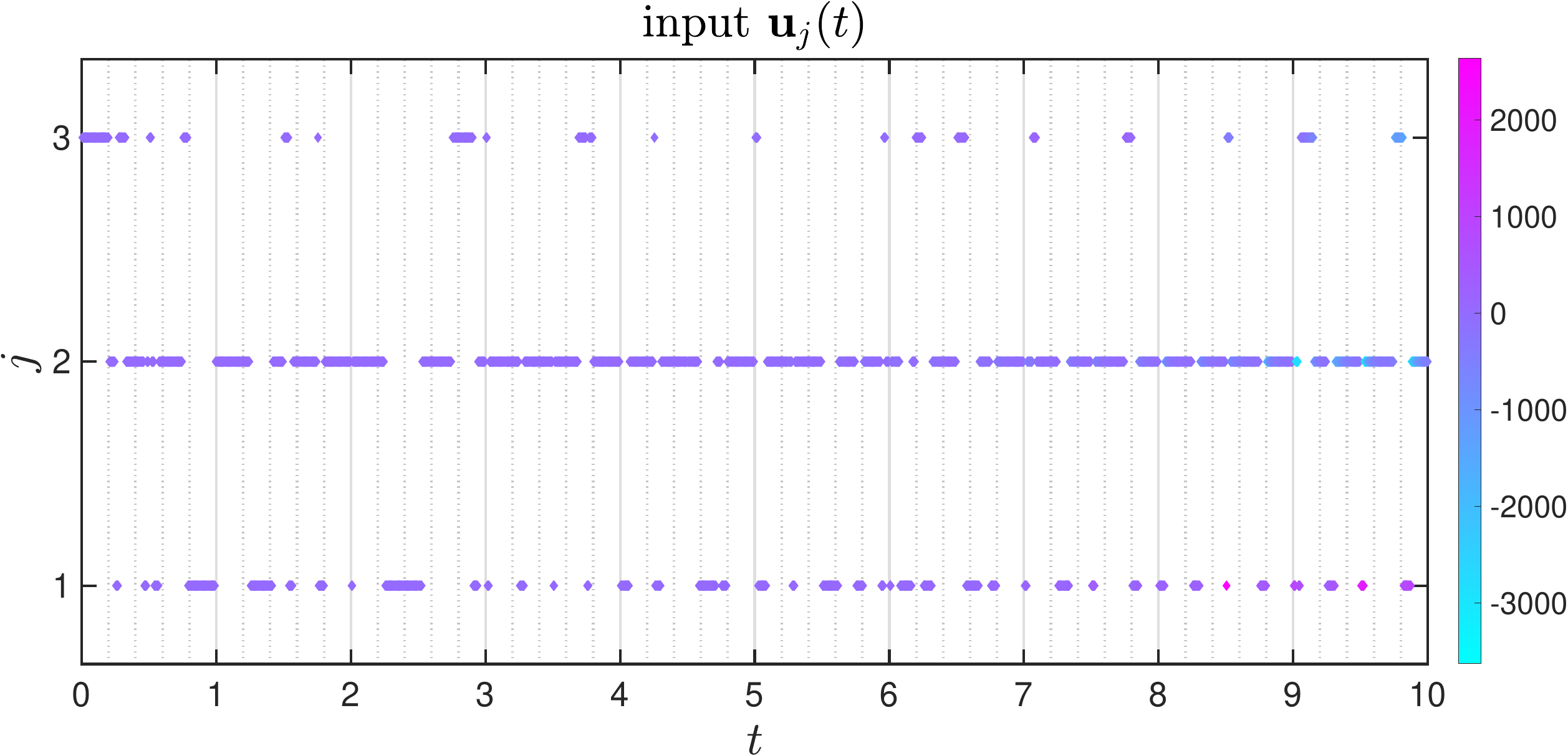}}
         \subfigure%[]
    { \includegraphics[height=.3\textwidth,width=.9\textwidth]{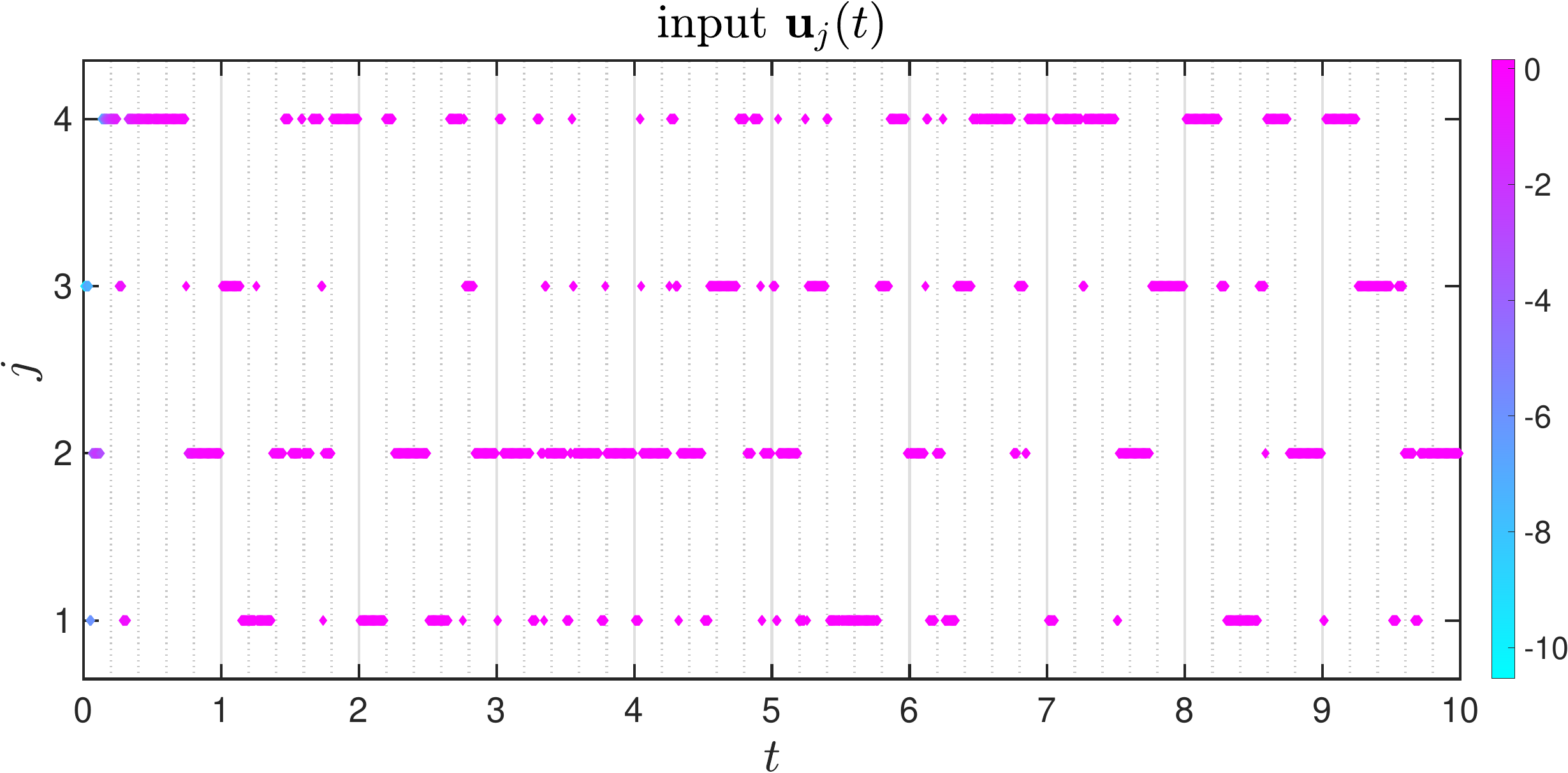}}
    \subfigure%[]
    {\includegraphics[height=.3\textwidth,width=.9\textwidth]{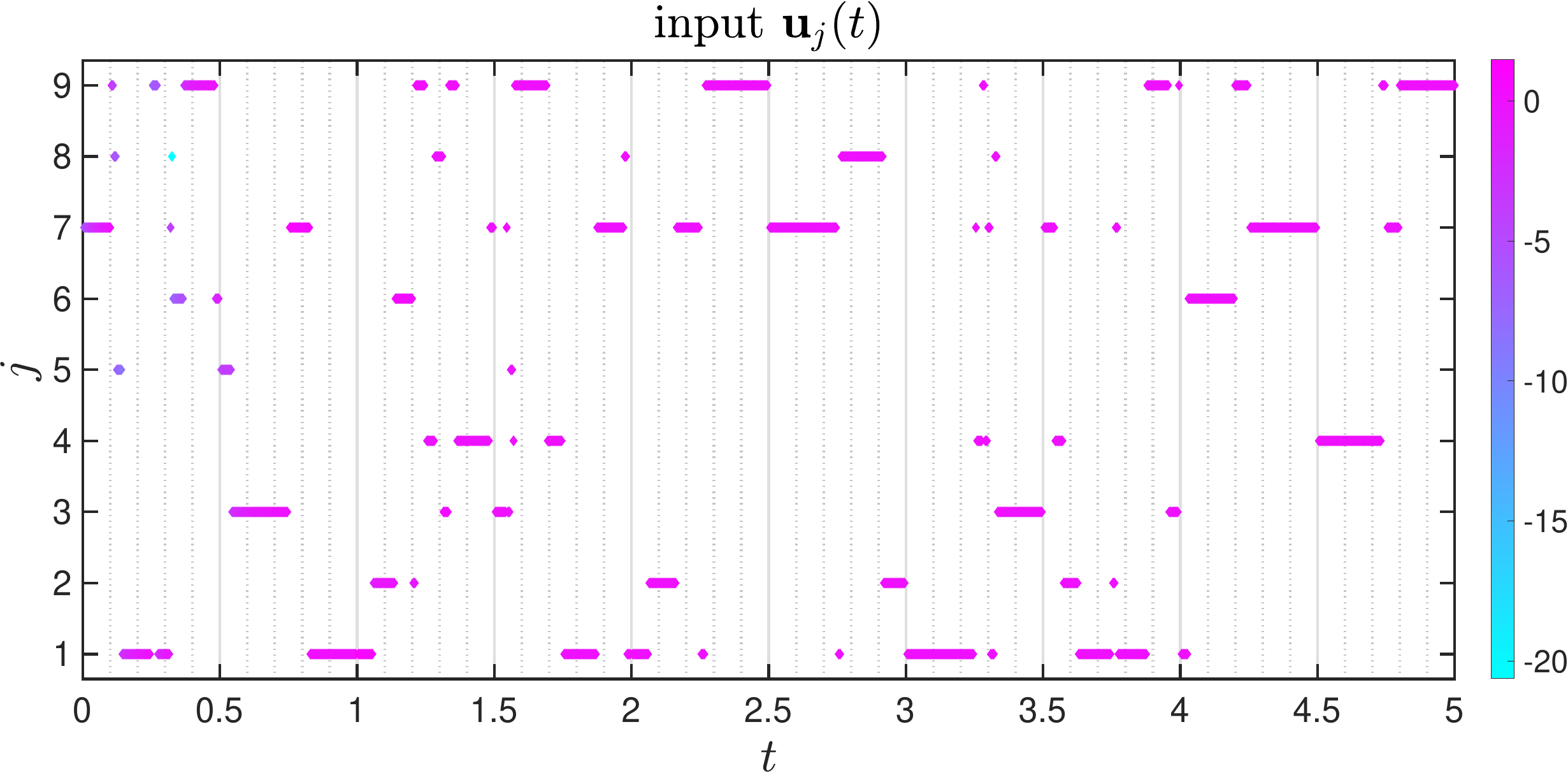}}
         \subfigure%[]
    { \includegraphics[height=.3\textwidth,width=.9\textwidth]{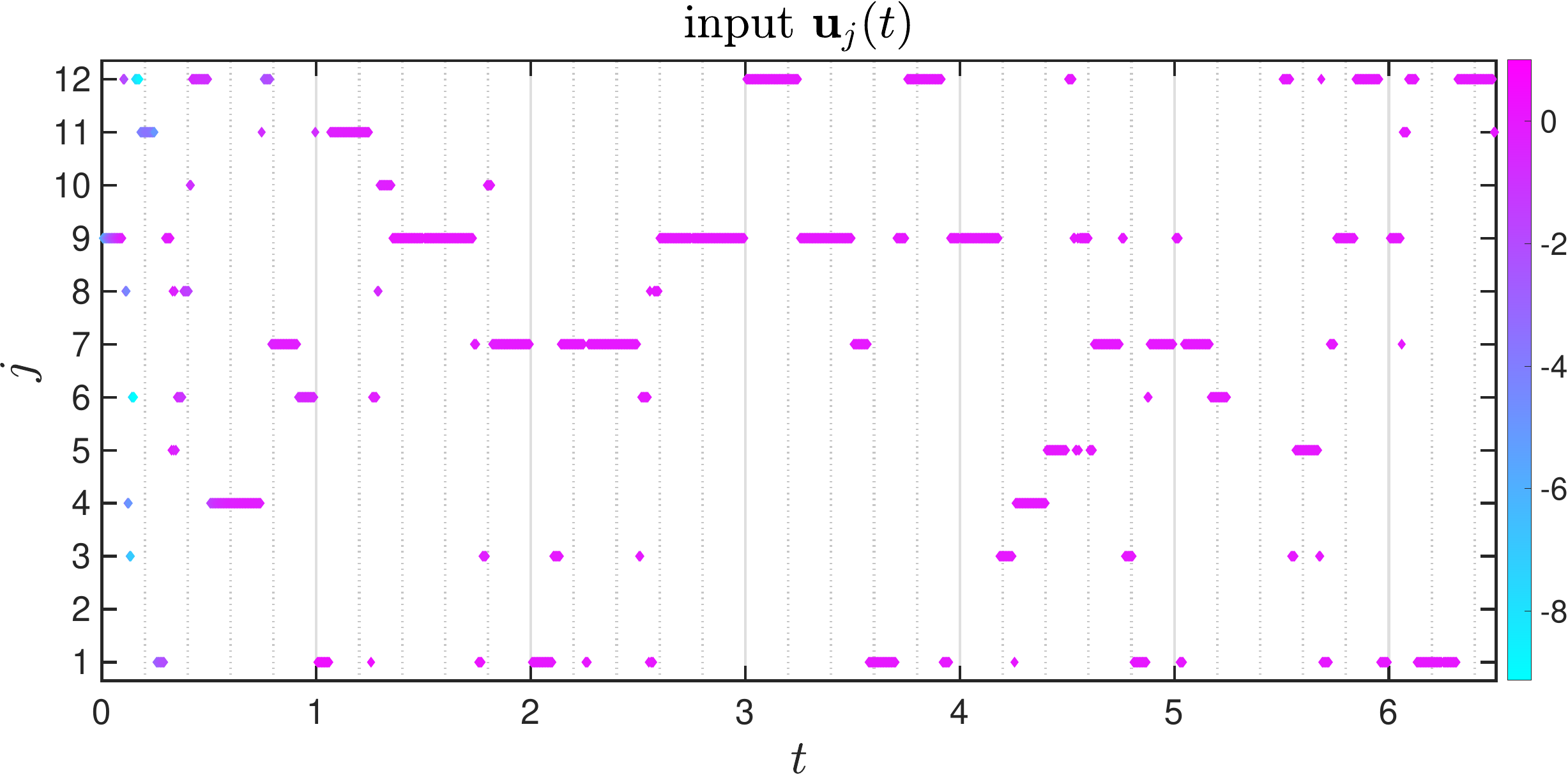}}
           \caption{Switching RHC for $M \in \{3, 4,9,12\}$ point actuators}\label{Fig:swicontrol3-12}
\end{figure}
The evolution of the magnitude~$\norm{\bfu_{\rm rh}(t)}{\ell^2}=\norm{u_{\rm rh}(t)}{\bbR}$ (see~\eqref{rhc-uc}) of the switching RHCs for cases $M \in \{ 3, 9\}$ is illustrated in Fig.~\ref{Fig:magcontrol_all}.  As we can see,  the magnitude  $\norm{u_{\rm rh}(t)}{\bbR}$ tends to infinity for the nonstabilizing case $M =3$,  whereas it converges to zero for the stabilizing case $M = 9$.  Similar convergence to zero has been observed for the cases $M \in \{ 4, 12\}$; thus, we did not illustrate these cases.
 \begin{figure}[htbp]
    \centering
            \subfigure%[]
    {\includegraphics[height=.3\textwidth,width=.9\textwidth]{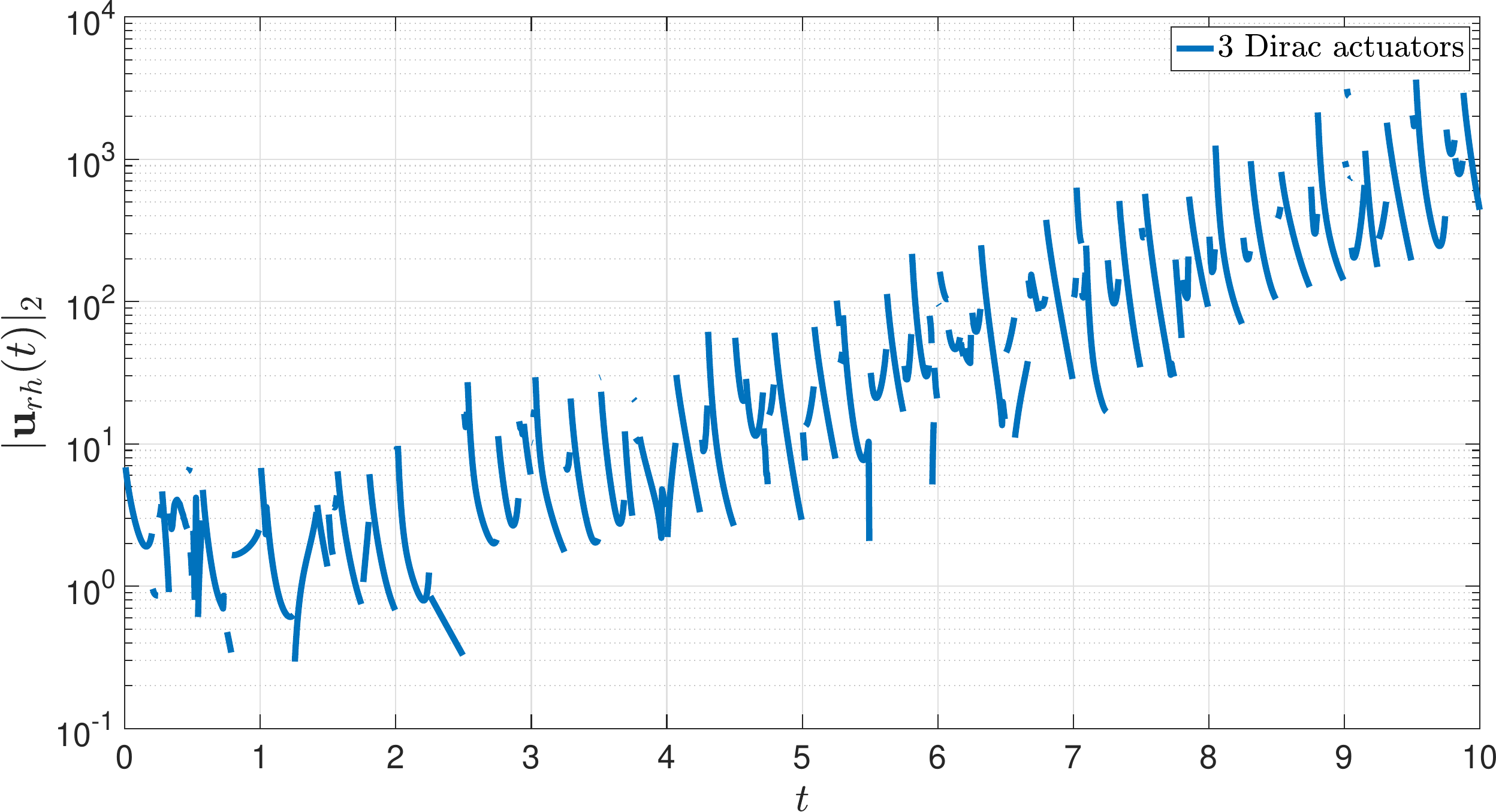} }
         \subfigure%[]
    {\includegraphics[height=.3\textwidth,width=.9\textwidth]{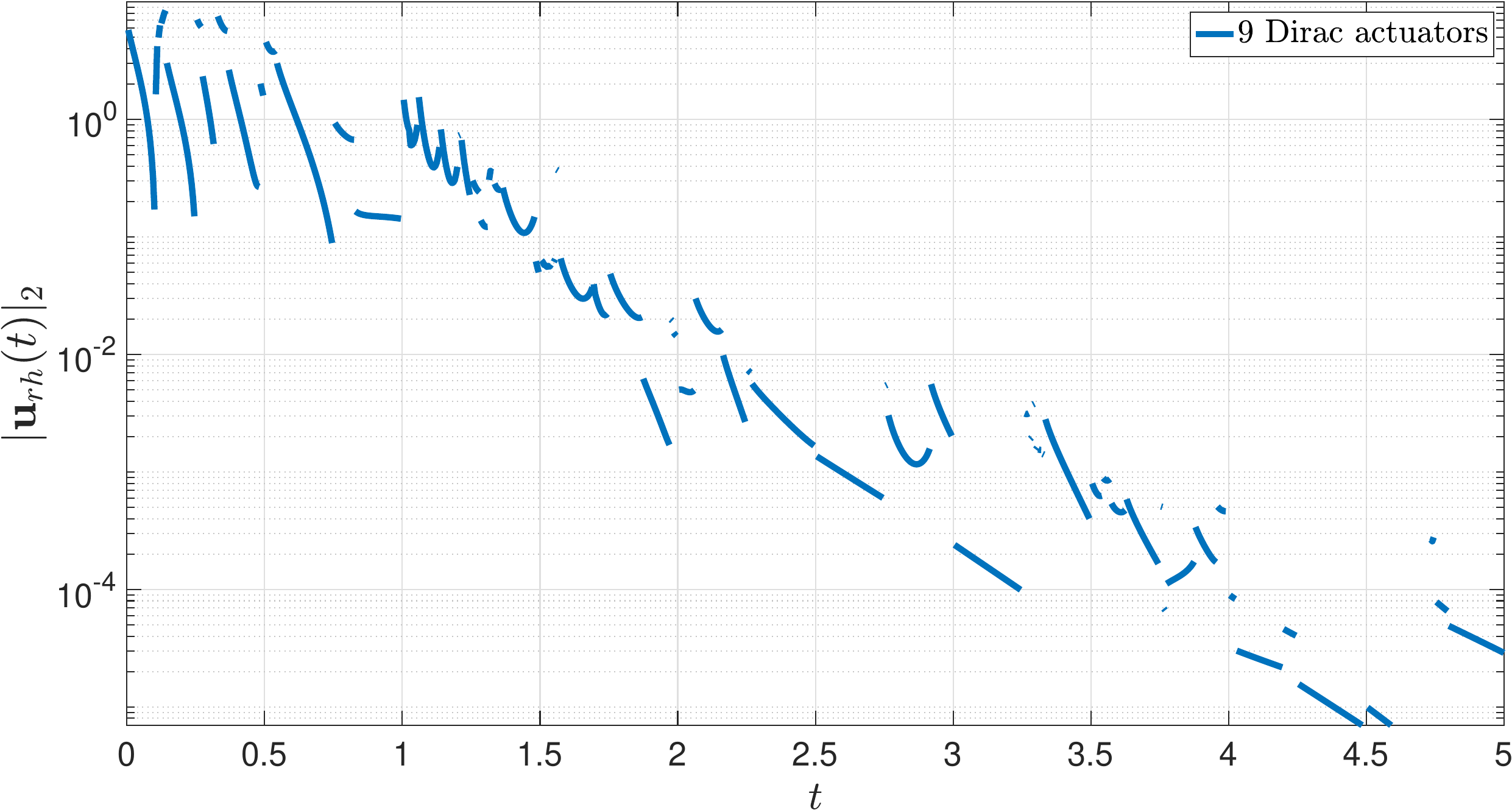} }
           \caption{Magnitude of RHC for~$M\in\{3,9\}$ point actuators}\label{Fig:magcontrol_all}
\end{figure}

Finally, in Table \ref{table1}, we report on the value of the performance index $J_{T_{\infty}}(\cdot; 0 ,y_0)$ (optimal truncated cost) and the $V'$-norm of the controlled state at the final computation time~${T_{\infty}=5}$. 
\begin{table}[htbp!]
\begin{center}
\scalebox{0.8}{
  \begin{tabular}{ | c | c | c | c | c | c |}
    \hline
     Number and placement of actuators & Uncontrolled  & $M = 3$  & $M = 4$ &$M = 9$ &$M = 12$\\
    \hline
    $  J_{T_{\infty}}(\cdot; 0,y_0) $ & $1.779 \times 10^{15}$ &$ 6.645 \times 10^{2}$ &$ 0.196$ & $0.378$& $0.260$ \\
    \hline
    $\norm{y(\cdot , T_{\infty})}{V'}$ &$1.504\times 10^{8}$ &$2.873 \times 10^{1}$&$ 0.001$  &$ 1.898\times10^{-5}$ &$1.658 \times 10^{-4}$ \\
    \hline
 \end{tabular}}
 \end{center}
 \caption{Numerical results for different placements for  $T_{\infty} = 5$ }
 \label{table1}
\end{table}
The results suggest that  tendentiously by taking a larger number of delta actuators we obtain a faster convergence to zero.  Note that this fact may depend also on  the placement of the actuators, for example, combining the results in Table~\ref{table1} and Fig.~\ref{Fig:stabil_all} we see that the case of~$M=9$ actuators seem to provide a rate of stability  similar to the case of~$M=12$ actuators; for large~$t$ the state corresponding to~$M=9$ is definitely closer to zero than the state corresponding to~$M=12$; and the cost corresponding to~$M=12$ is smaller than the one corresponding to~$M=9$. Therefore, it is debatable which among~$M\in\{9,12\}$ provides the ``best'' performance.

In conclusion, the results show that the number of actuators play crucial role in the stabilizing performance of the control, and strongly suggest that the placement of the actuators is important as well.

%%%%%%%%%%%%%%%%%%%%%%%%%
%%%%%%%%%%%%%%%%%%%%%%%%%
\subsection{On switching and nonswitching controls}
\label{exp2}
We focus on the case of~$M=4$ actuators,  placed as in Fig.~\ref{Fig:meshAct}.
The purpose of Fig.~\ref{Fig:comp-swi-nonswi} is to compare the performance of the switching RHC with the 
nonswitching analogue,  where all of the actuators can be active simultaneously.  This is motivated by the theoretical strategy where a (not necessarily optimal) switching control was constructed from a (not necessarily optimal) nonswitching one.
\begin{figure}[htbp]
    \centering
   { \includegraphics[height=0.35\textwidth,width=0.5\textwidth]{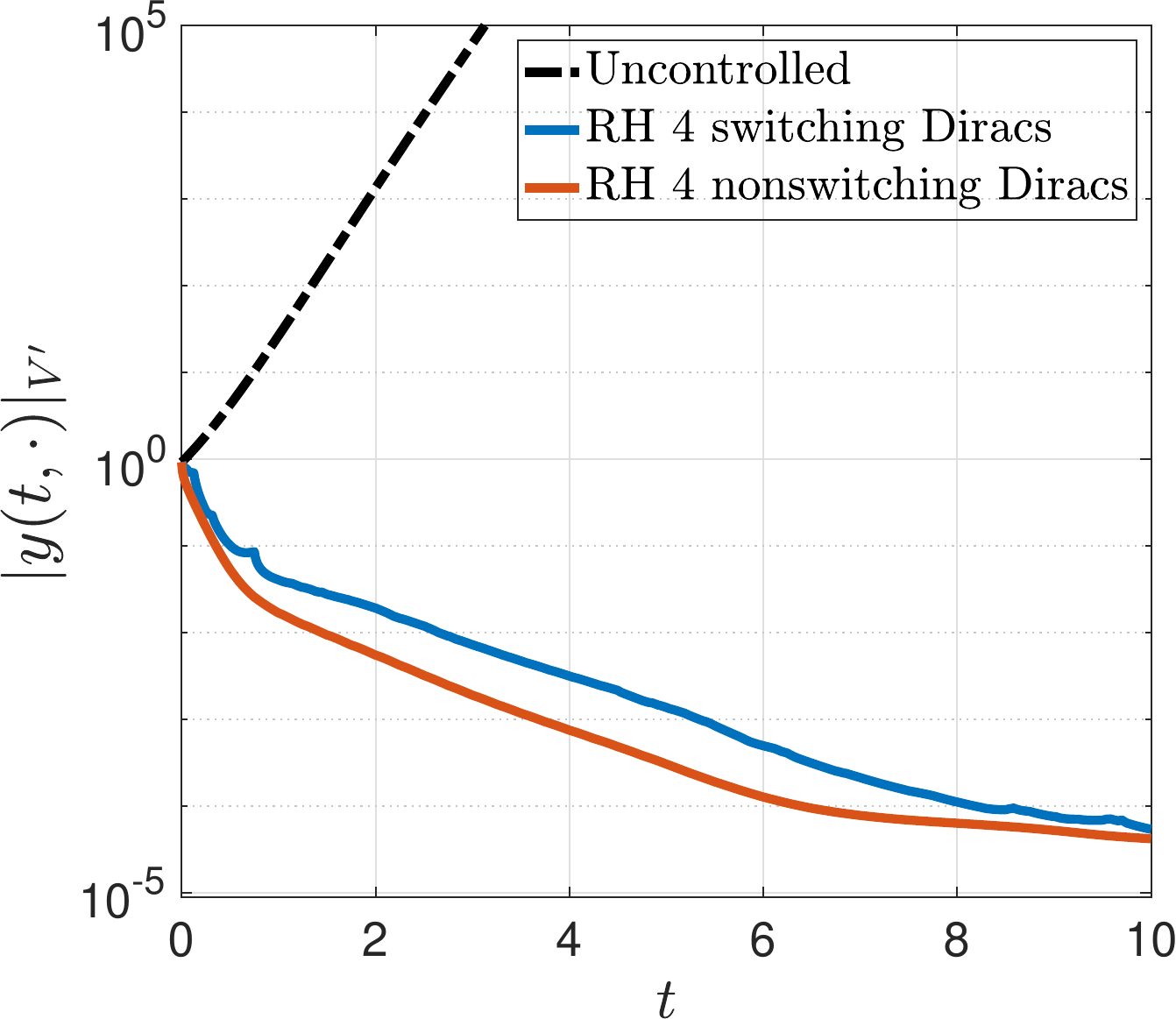} }
 \caption{Evolution of $\norm{y(t,\cdot)}{V'}$}\label{Fig:comp-swi-nonswi}
\end{figure}
{We can see that the performance of the two controls are relatively close, which suggests that the action of the  switching RHC approximates the action of the  nonswitching RHC one.} These facts are also validated by the results given in Table \ref{table2}, reporting the truncated cost and norm of the final state. Finally, to
 \begin{table}[htbp!]
\begin{center}
\scalebox{0.8}{
  \begin{tabular}{ | c | c | c | c |}
    \hline
     Number and placement of actuators & Uncontrolled  & Switching for $M = 4$  & Nonswitching for $M = 4$ \\
    \hline
    $  J_{T_{\infty}}(\cdot; 0,y_0)$  & $1.5\times 10^{32}$ &$ 0. 196 $&$ 0.099$ \\
    \hline
    $\norm{y(\cdot , T_{\infty})}{V'}$ &$4.354\times 10^{16}$ &$5.452 \times 10^{-5} $&$ 4.243 \times 10^{-5}$ \\
    \hline
 \end{tabular}}
 \end{center}
 \caption{Numerical results for $4$ actuators for $T_{\infty} = 10$}
 \label{table2}
\end{table}
gain some intuition, Fig.~\ref{Fig:timesnapshots} gathers time-snapshots of the switching RHC and corresponding state. We can see the effect on the state around the location of the active actuator, at given time instants.
\begin{figure}[htbp]
    \centering
    \subfigure[$t = 0.02$]
    {\includegraphics[height=0.21\textwidth,width=0.51\textwidth]{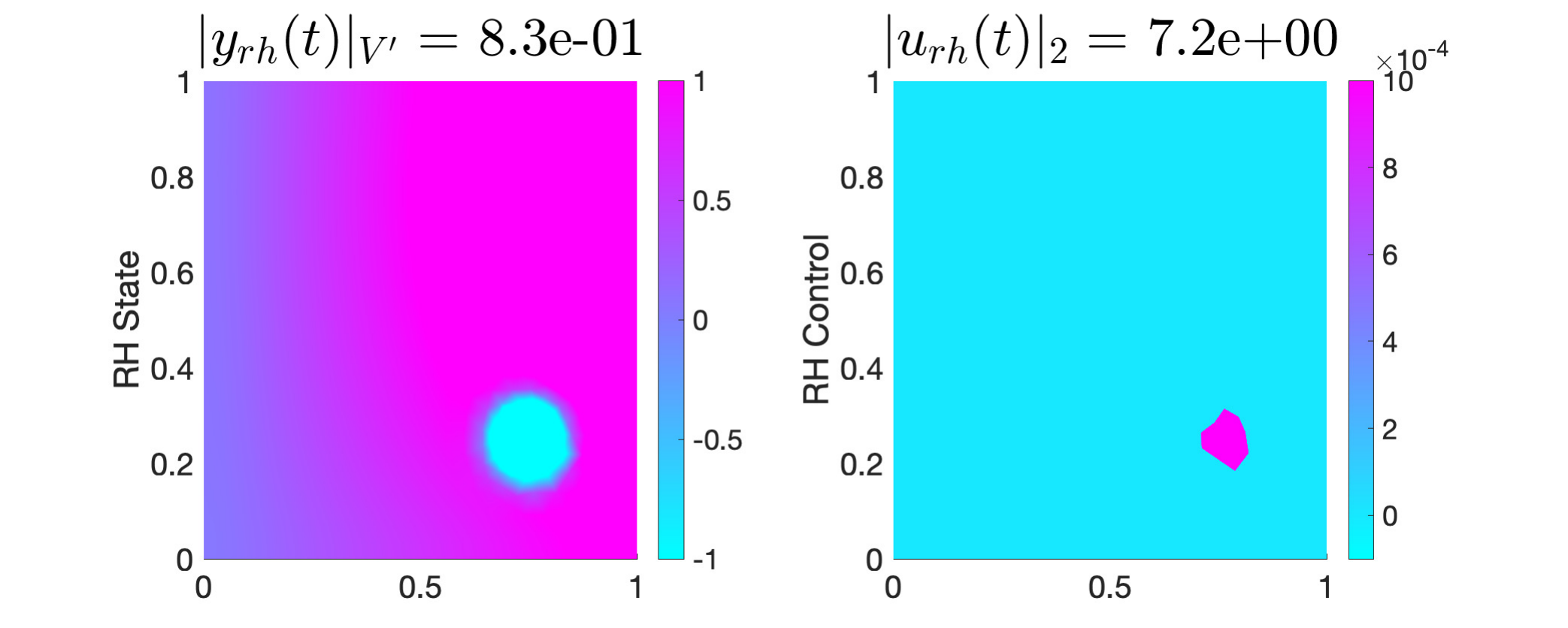}}
    \hspace{-1.5em}
    \subfigure[$t = 0.05$]
    {\includegraphics[height=0.21\textwidth,width=0.51\textwidth]{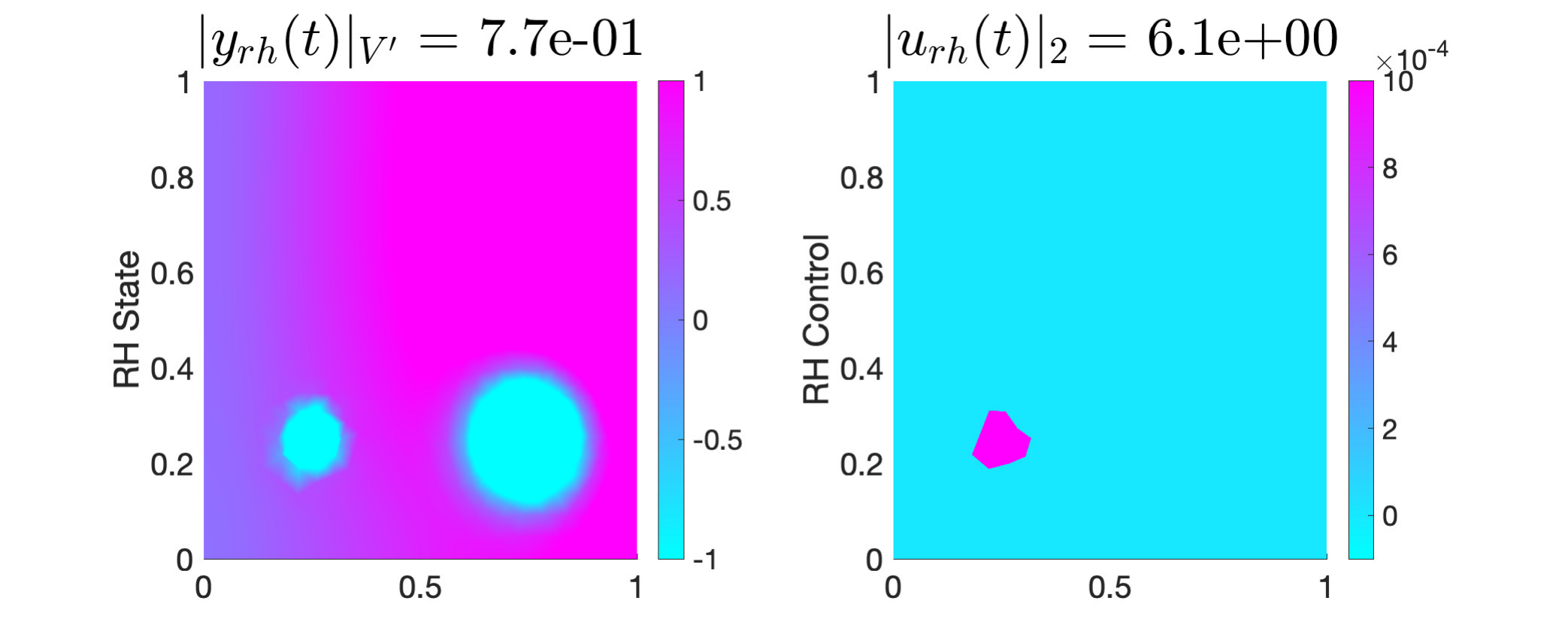}}
    \\
    \subfigure[$t = 0.2$]
    {\includegraphics[height=0.21\textwidth,width=0.51\textwidth]{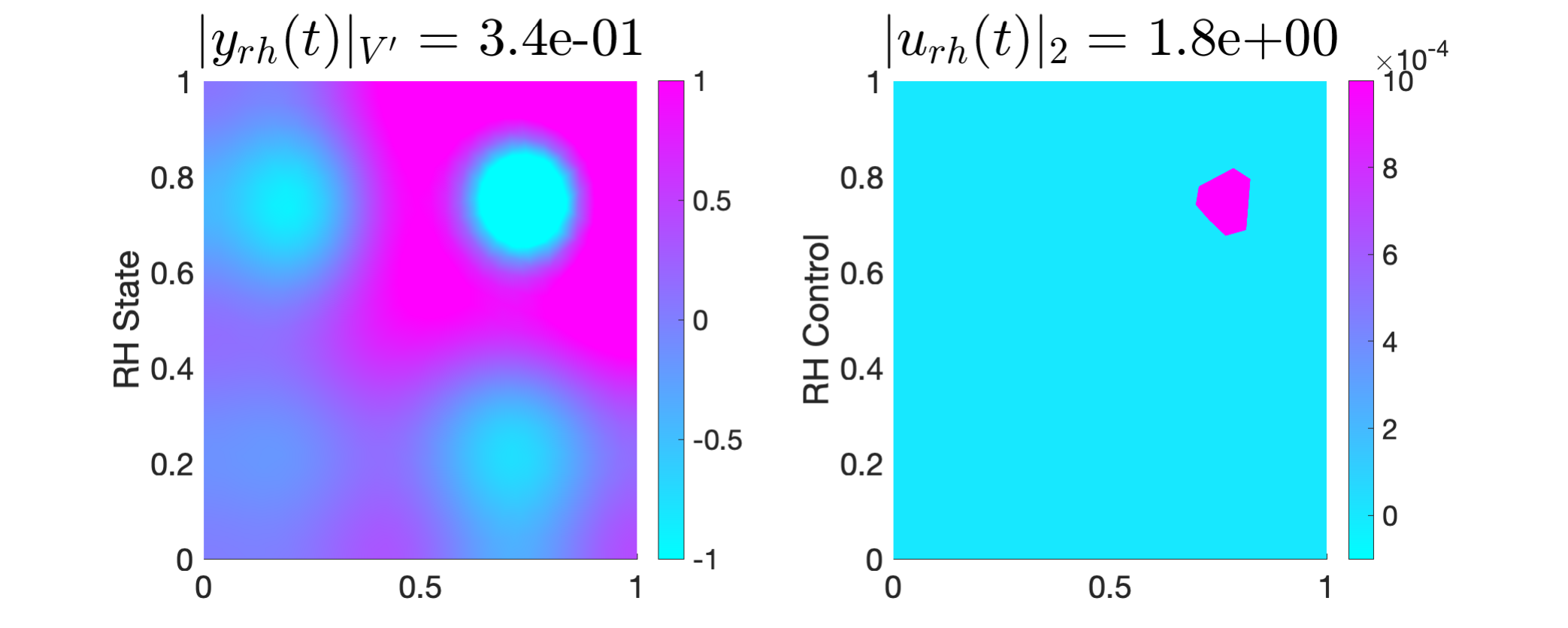}}
    \hspace{-1.5em}
     \subfigure[$t = 0.3$]
    {\includegraphics[height=0.21\textwidth,width=0.51\textwidth]{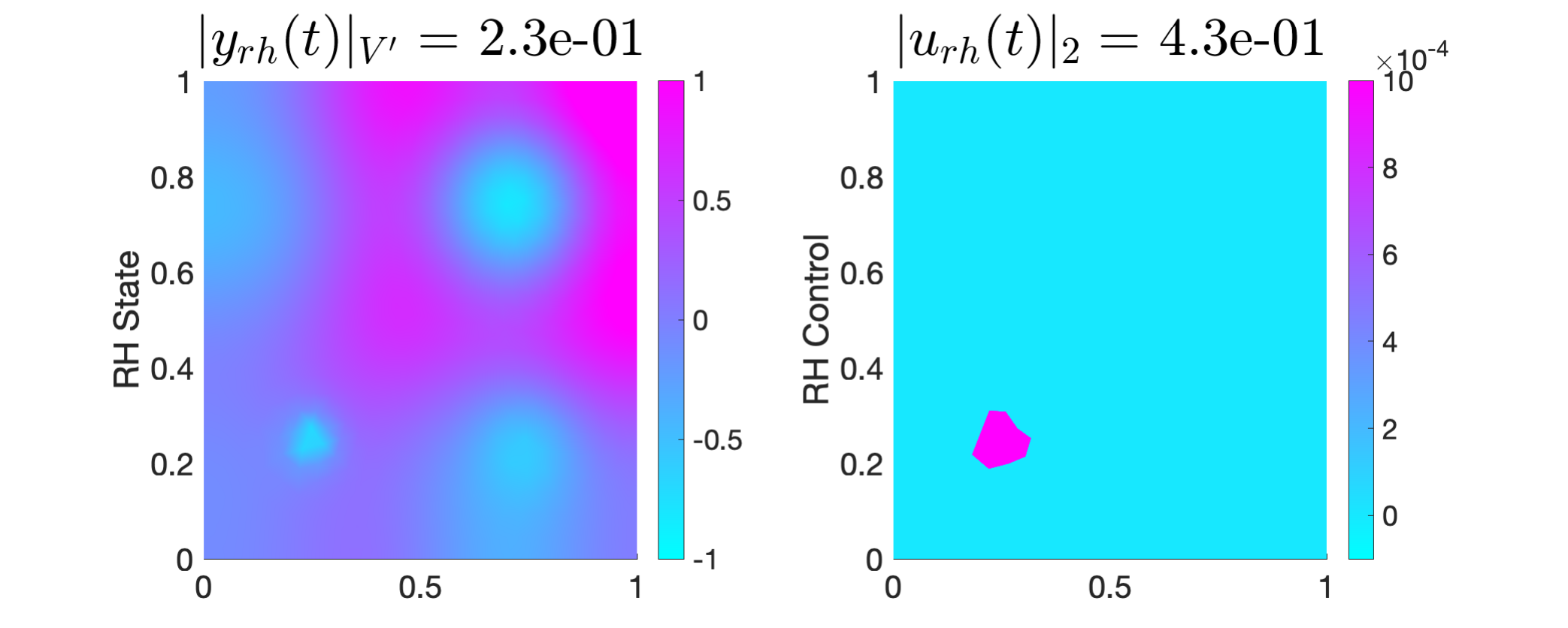}}
     \\
     \subfigure[$t = 0.35$]
    {\includegraphics[height=0.21\textwidth,width=0.51\textwidth]{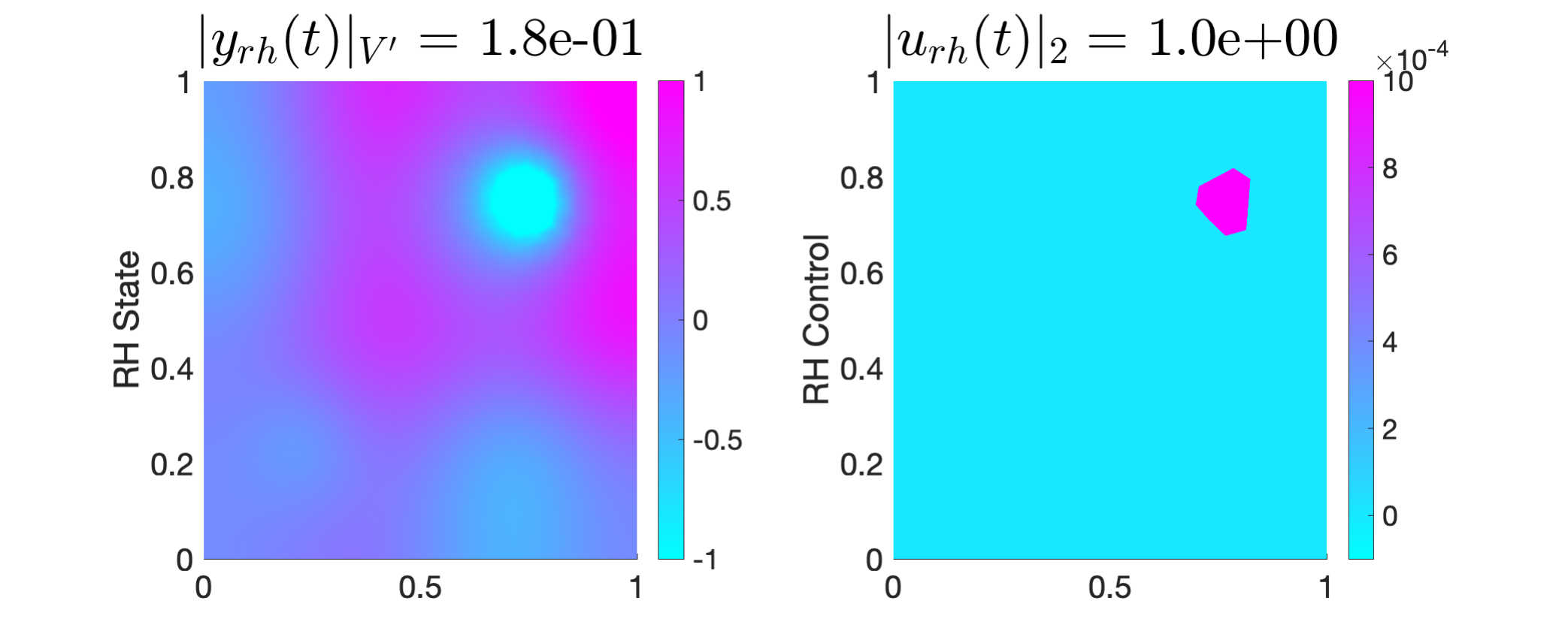}}
    \hspace{-1.5em}
    \subfigure[$t = 0.80$]
    {\includegraphics[height=0.21\textwidth,width=0.51\textwidth]{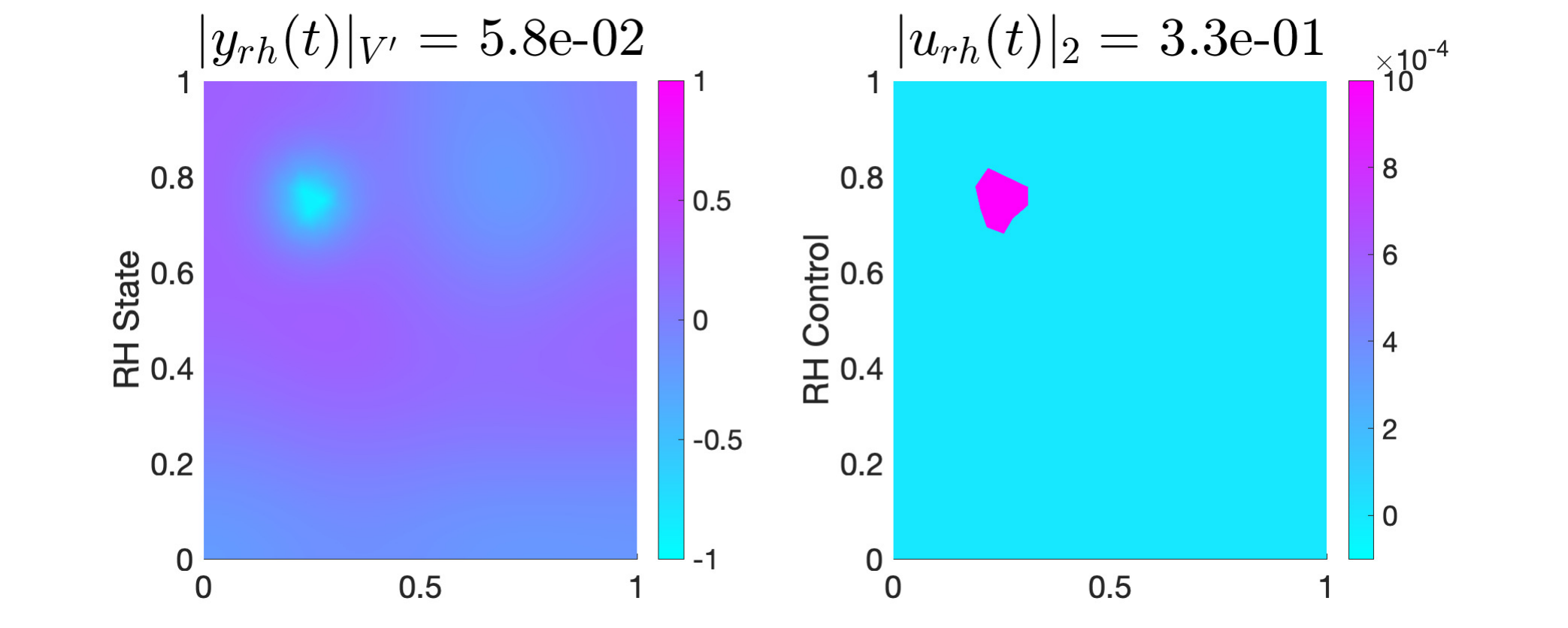}}
    \\
     \subfigure[$t = 5.20$]
    {\includegraphics[height=0.21\textwidth,width=0.51\textwidth]{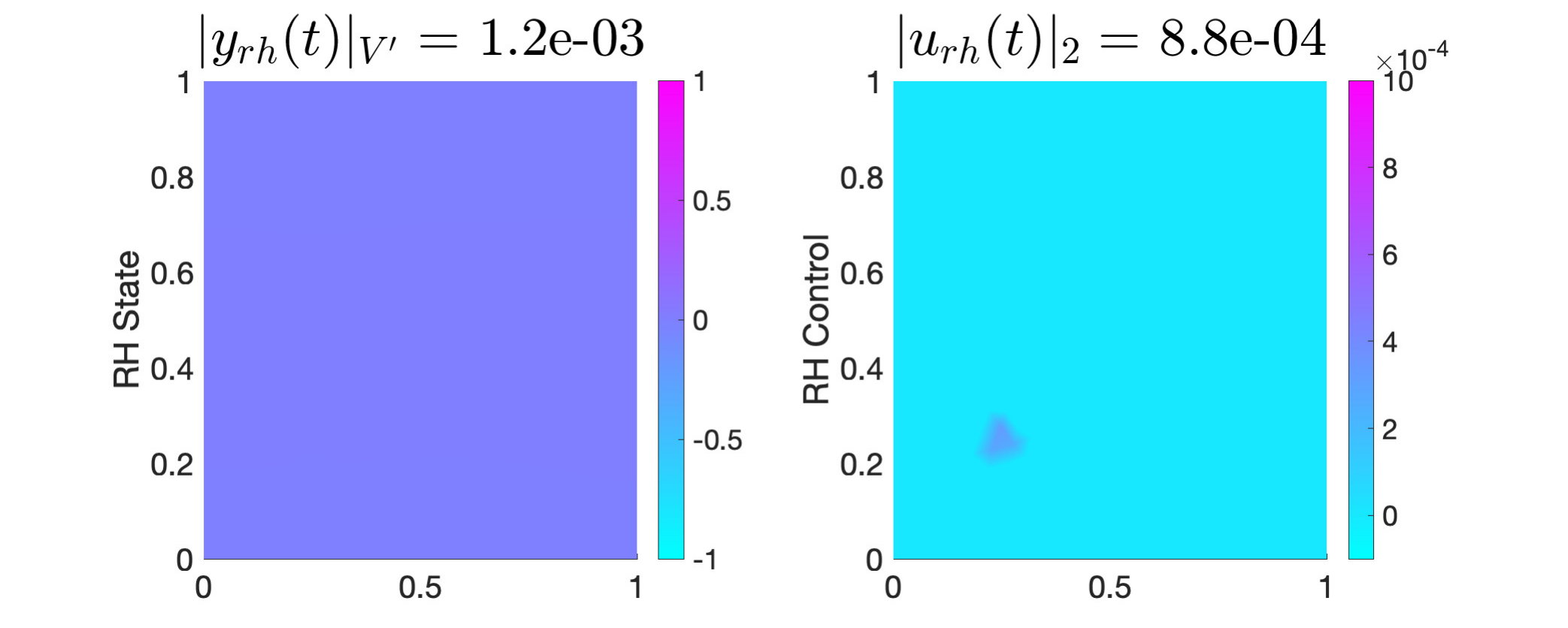}}
   \hspace{-1.5em} 
     \subfigure[$t = 9.95$]
    {\includegraphics[height=0.21\textwidth,width=0.51\textwidth]{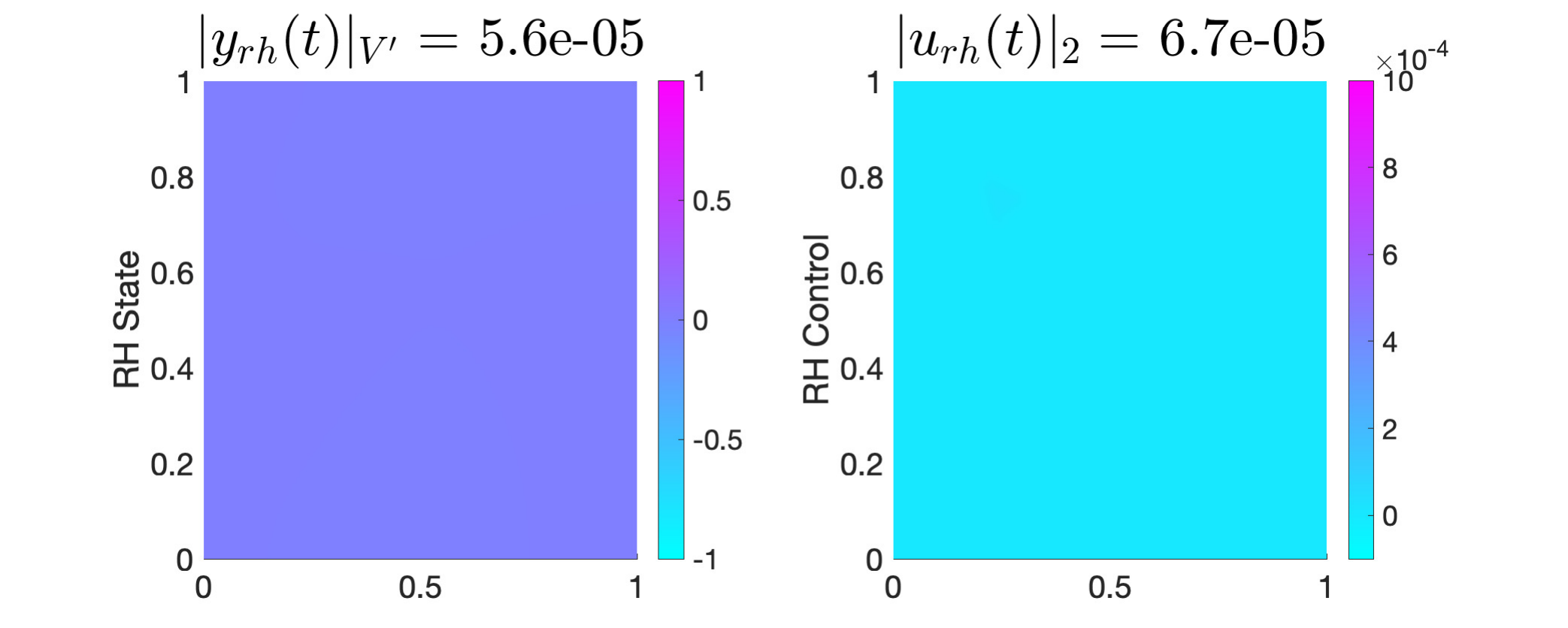}}
    \caption{Time-snapshots of the receding horizon state and control; $M=4$.}
    \label{Fig:timesnapshots}
\end{figure}

%%%%%%%%%%%%%%%%%%%%%%%%%
%%%%%%%%%%%%%%%%%%%%%%%%%
%%%%%%%%%%%%%%%%%%%%%%%%%
\section{Final comments}\label{S:finalcomm}
Departing from a stabilizability result in~\cite{KunRodWal24-cocv} by means of control forcings given by linear combinations of a finite number~$M$ of delta-distribution actuators~$\deltafun_{x^j}$, we have used the continuity of the solutions in a so-called relaxation metric~\cite{Gamk78} and tools from optimal control to show the existence of a stabilizing (open-loop) control switching between those actuators~$\deltafun_{x^j}$.
We analyzed the performance and computation of a receding horizon (model predictive) control which involves the admissible set of controls~$\clU_{\rm ad}^{\overline{t}_0,T}$ and employs a  cardinality constraint to guarantee that only one actuator is active at each time instant. In order to demonstrate the theoretical findings we presented results of simulations showing the stabilization performance of the computed switching RHC.

Both the number~$M$ and location set~$\bfx=\{x^j\in\Omega\mid 1\le j\le M\}$ of actuators play an important role on the stabilization performance of the switching control. The investigation of a/the {\em ``best'' location set} of actuators is an interesting subject for future work.

Other interesting subjects concern  finite horizon optimal control problems and the computational performance of the iterative procedure within Step~\ref{alg:solve} in Algorithm~\ref{RHA}; see~\eqref{proj_Ik}. Namely, can we guarantee the  {\em existence of  minimizers} for these problems? Can we improve the {\em computational performance} of the iterative procedure? If necessary, for example, by considering a different cost functional or a different  projection onto the admissible set~$\clU_{\rm ad}^{\overline{t}_0,T}$?

%%%%%%%%%%%%%%%%%%%%%%%%%
%%%%%%%%%%%%%%%%%%%%%%%%%
%%%%%%%%%%%%%%%%%%%%%%%%%
\bigskip\noindent
{\bf Acknowledgments.}
S. Rodrigues acknowledges partial support from State of Upper Austria and Austrian Science
Fund (FWF): P 33432-NBL.

 \bibliographystyle{plainurl}
 \bibliography{ActuatorMoving}

\end{document}